\documentclass[11pt]{amsart}
\usepackage{graphicx, xcolor}
\usepackage{amssymb}
\usepackage{verbatim}

\numberwithin{equation}{section}

\def\R{\mathbb{R}}

\renewcommand\d{\partial}

\def\e{\varepsilon}
\def\e{\varepsilon}

\newcommand\br{\begin{rem}}
\newcommand\er{\end{rem}}
\newcommand\bp{\begin{pmatrix}}
\newcommand\ep{\end{pmatrix}}
\newcommand\be{\begin{equation}}
\newcommand\ee{\end{equation}}
\newcommand\ba{\begin{equation}\begin{aligned}}
\newcommand\ea{\end{aligned}\end{equation}}

\newtheorem{theorem}{Theorem}[section]
\newtheorem{proposition}[theorem]{Proposition}
\newtheorem{lemma}[theorem]{Lemma}
\newtheorem{corollary}[theorem]{Corollary}

\newtheorem{remark}[theorem]{Remark}

\title[Traveling waves for quasilinear Fisher-Burgers equations]{Heteroclinic traveling fronts for a generalized Fisher-Burgers equation
with saturating diffusion} 

\begin{document}

\maketitle

\begin{center}
MAURIZIO GARRIONE\footnote{Universit\`a di Milano Bicocca, Dipartimento di Matematica e Applicazioni, Milano (Italy). E-mail address:
\texttt{maurizio.garrione@unimib.it}}, MARTA STRANI\footnote{Universit\'e Paris Diderot, Institut de Math\'emathiques de Jussieu, Paris
(France), E-mail addresses: \texttt{martastrani@gmail.com; marta.strani@imj-prg.fr}}\\
\end{center}

\vskip1cm

\begin{abstract}
We study the existence of monotone heteroclinic traveling waves for a general Fisher-Burgers equation with nonlinear and possibly
density-dependent diffusion. Such a model arises, for instance, in physical phenomena where a saturation effect appears for
large values of the gradient. We give an estimate for the critical speed (namely, the first speed for which a monotone heteroclinic traveling wave
exists)
for some different shapes of the reaction term, and we analyze its dependence on a small real parameter when this brakes the diffusion, complementing our study with some numerical simulations.
\end{abstract}

\begin{quote}\footnotesize\baselineskip 14pt 
{\bf Key words:} saturating diffusion, Fisher-Burgers equation, density-dependent diffusion, traveling waves, admissible speeds. 
 \vskip.15cm
\end{quote}

\begin{quote}\footnotesize\baselineskip 14pt 
{\bf AMS 2010 subject classification:} 34C37, 35K55, 35K57. 
 \vskip.15cm
\end{quote}


\pagestyle{myheadings}
\thispagestyle{plain}

\section{Introduction}\label{intro} 

This paper concerns one-dimensional quasilinear reaction-diffusion PDEs. We study
the existence of monotone heteroclinic connections of traveling-wave type for
reaction-diffusion equations with saturating and density-dependent diffusion, in presence of a convective term;
precisely, we deal with the equation
\begin{equation}\label{QFB}
\d_t u=\d_x(P[\d_x(D(u))]) - \d_x h(u) + f(u), \quad (u=u(t, x)),
\end{equation}
under a suitable set of assumptions 
(see Section \ref{sez2}). 
Not to overload the notation and the
readability of the paper, our main results will focus on the special case when $P(s)$ is of mean curvature type, namely 
\begin{equation}\label{curvat}
P(s)=\frac{s}{\sqrt{1+s^2}} 
\end{equation}
(see, for instance, \cite{KurRos}). However, the presented results
may hold true for a larger
class of operators, as it will be pointed out later on in the paper.
\smallbreak
Reaction-convection-diffusion type models like \eqref{QFB}, with a special focus on traveling waves, are very popular and can be used to describe
different mathematical and physical phenomena: the main areas of application, for which we quote a few references, are, for instance, neurophysics and biophysics \cite{Kee80, OstYan75}, chemical physics \cite{Fif76b, OrtLos75}, population genetics
\cite{AroWei75,LiLiWang}, tumor growth \cite{PerTanVau14} and mathematical ecology \cite{MimMur79}. As describing quite complicated
realistic models,
most of the above applications concern systems of equations in higher dimensional spaces.
\smallbreak
Inside this vast class, the Fisher-Kolmogorov-Petrovsky-Piskounov (FKPP) equation
\begin{equation}\label{F}
\d_t u =\d_x^2 u + u(1-u),
\end{equation}
introduced by Fisher \cite{Fis37} to model the
spread of an advantageous gene inside a $1$-dimensional
population, and the viscous Burgers equation \cite{Burgers}
\begin{equation}\label{VB}
\d_t u=\d_x^2 u - u \d_x u,
\end{equation}
finding application, for instance, in problems related to gas dynamics and traffic
flows, are probably among the most popular.
In the recent years, an increasing interest has been devoted, as well, to a ``combination'' of \eqref{F} and \eqref{VB}, called
\emph{Fisher-Burgers} (or \emph{Burgers-Fisher}) equation, which is obtained by setting $P(s)=s$, $D(u) \equiv u$,
$f(u)=u(1-u)$ and
$h(u)=u^2/2$ in equation \eqref{QFB}; this is perhaps the easiest example of PDE taking into account both the
reaction-diffusion and the convection effects. Such an equation has been widely studied 
by means of both theoretical and numerical methods (see, e.g., \cite{Mac13,ZLZ13}), and represents the starting point of our investigation. 
In order to better set up our scheme and
introduce our results for generalized Fisher-Burgers equations displaying
saturating effects, we thus now start briefly reviewing
the results which are available, in our perspective, for equations \eqref{F} and \eqref{VB}. 
\smallbreak
After its introduction \cite{Fis37} (as pointed out, for instance, in \cite{Fen07}, actually the investigation and analysis of traveling waves in chemical reactions is probably first due to Luther \cite{Lut06}), the Fisher equation \eqref{F} was widely studied, from a mathematical point of view, starting with the
paper \cite{KolPetPis37}. As mentioned above, it was meant to describe the evolution of the relative concentration of a gene inside a
population; under the natural assumption $f(0)=0=f(1)$ (meaning that no reaction can be present if the gene is completely spread or not
spread at all), it was already
observed in these first articles the natural appearance of heteroclinic solutions connecting the equilibria $0$ and $1$ - we will call them
\emph{$\{0, 1\}$--connections} - of traveling wave type, embodying the
transition from the situation where the gene is not present to the one where every individual possesses it.
By standard phase-plane arguments, it is immediate to see (see, for instance, \cite{FifMcL}) that such solutions are
necessarily monotone; it was then proved in \cite{FifMcL} that they also have some stability properties for a certain
class of initial data. 
\newline
For a better comprehension of the model \eqref{F}, it is thus natural to search first for solutions having the form $u(t,
x)=v(x+ct)$, with $v$ an increasing
function such that $v(-\infty)=0$, $v(+\infty)=1$. The values of $c$ yielding such solutions are called
\emph{admissible speeds}, and it turns out
that the set of the admissible speeds for equation \eqref{F} is an unbounded interval $[c^*, +\infty)$; $c^*$ takes the
name of \emph{critical speed}. For example, in the Fisher case \eqref{F} it is
$c^*=2$. Later on, reaction terms other than the Fisher one were studied: the most popular are the so called
\emph{type B} and \emph{type C} (according to the terminology in \cite{BerNir}). The former has been considered in some combustion
problems, while the latter appears naturally in Biology (see \cite{AroWei}), in relation with the death rates of
populations presenting different genotypes with
respect to a particular gene, or in aggregative-diffusive models for population dynamics (see, e.g.,
\cite{MalMarMat07} for the case of a sign-changing diffusion). By some monotonicity arguments, for both of these situations (see
Section
\ref{sez4} below) it may be seen
that the admissible speed exists and is unique.  
\smallbreak
On the other hand, the viscous Burgers equation \eqref{VB} often appears as a simplification of more complicated models: a well known
example is the one given by the
Navier-Stokes equations, describing the dynamics of compressible or incompressible fluids.
If we consider \eqref{VB} without the second order viscous term and with general flux function describing the convection phenomenon, it is
well known that solutions with
discontinuities (called internal 
shocks) are allowed in the class of the so called  {\it entropy solutions}; in particular, in order to provide uniqueness, it is common to
assume the flux function to be convex (for more details, see \cite{Evans}).
On the other hand, when the viscosity is ``switched on",  all the discontinuities that appear
at the hyperbolic level turn out into smooth transition layers. In this setting, smooth traveling wave solutions to the Cauchy problem 
associated with \eqref{VB} have been extensively studied over the past years; indeed, when looking for a solution of the form
$u(t, x)=v(x+ct)$, equation \eqref{VB} can be explicitly solved by integration, leading to
\begin{equation}\label{TVB}
v(x+ct)= \frac{u_- +u_+e^{\kappa(x+ct)}}{1+e^{\kappa(x+ct)}},
\end{equation}
where $u_\pm := -c\mp \sqrt{c^2+2A}$ ($A$ being an integration constant), and $\kappa:=( u_--u_+)/2$.  In particular, $u_->u_+$ by
definition,
so that $\kappa>0$ and \eqref{TVB} represents  a 
decreasing function connecting the values $u_-= \lim_{x\to -\infty} u(t, x)$ and $u_+ = \lim_{x \to +\infty} u(t, x)$. Finally, it
holds that 
\begin{equation}\label{condC}
c={-}\frac{u_++u_-}{2}
\end{equation}
meaning that, once the values $u_-$ and $u_+$ are chosen, there exists a unique $\{u_-, u_+\}$--connection of traveling wave type, having
speed $c$. 
\medbreak 
What we have described so far holds true when the diffusion is \emph{linear}. However, assuming that a diffusive
process in nature takes
place in a homogeneous way, disregarding local biological and physical conditions (both environmental - like light, humidity, temperature
- and intrinsic in the behavior of the described population - e.g., its density or concentration, its gradient, cf., for instance, \cite{Tur2015}), seems to be quite a
strong restriction. For instance, it may be more likely to imagine that individuals move from richly to scarcely populated areas, or, on the
contrary, that under some conditions they tend to aggregate (for instance under the risk of extinction), in opposition to a diffusion-type
phenomenon. Similarly, it may be expected that the diffusion relents once the gradient of the density becomes large, in view of
a kind of blockage effect at too high gradients. 
\newline
Of course, there are many ways of introducing such effects in a model, leading to different pictures for the considered problem; many authors worked in this direction, part of the associated bibliography being available inside the references we
mention at
the end of the paper. Our choice in the present article is to take into account a usual diffusion process (towards less
``populated'' areas) which depends on the population density and displays a saturation effect for large gradients;
explicitly, motivated, among the others, by the papers \cite{Had, MalMar02} and \cite{KurRos,
Ros}, we consider a diffusion depending
on $\d_x u$ through the mean-curvature type operator (as the one considered in \cite{GarSan}) and possibly density-dependent.
\newline
In any case, we warn the reader that modifying the assumptions in the
diffusion makes possible, in principle, the appearance
of other kinds of solutions which are not  regular any more, since the functional space to appropriately set
the problem into may
include discontinuous functions (for instance, in the case \eqref{curvat} one is led to work in the space of bounded variation functions).
However, we here disregard this issue, focusing since the beginning on regular traveling waves. 
\smallbreak
As for the admissible speeds, and referring to the previous discussion about the Fisher equation, introducing a
density-dependent and saturating diffusion may lead to some different outcomes with respect to the linear case. If, on one hand, the case of Fisher-type nonlinearities like $f(u)=u(1-u)$ essentially displays the same features, on the other hand it may be seen
that the situation changes when the reaction term is of type C: in the recent paper \cite{GarSan} this was highlighted
when dealing with the model 
\begin{equation*}
\d_t u= \d_x (P(\d_x u)) +f(u),
\end{equation*}
namely equation \eqref{QFB} with $D(u)=u$ and $h'(u) \equiv 0$. 
\newline
Let us also notice that, wishing to deal instead with a similar issue for the quasilinear counterpart of the Burgers equation
\begin{equation}\label{QB}
\d_t u= \d_x (P(\d_x u)) - \d_x h(u), 
\end{equation}
the situation is slightly
different from the one for \eqref{F}, since equation \eqref{QB} admits all the constants as equilibria. 
Fixed two equilibria $u_\pm$ to be connected, in order to find monotone $\{u_-, u_+\}$-connections of traveling wave type
$u(t, x)=v(x+ct)$ one thus has to solve the boundary value problem 
\begin{equation}\label{BIIord}
\left\{
\begin{array}{l}
(P(v'))' - (h'(v)+c)v' = 0 \\
v(-\infty)=u_-, \; v(+\infty)=u_+, 
\end{array}
\right.
\end{equation}
asking $v'(s) > 0$ for every $s \in \mathbb{R}$ (or $v'(s) < 0$ for every $s \in \mathbb{R}$). 
Integrating the differential equation on the real line and noticing that $v'(\pm \infty)=0$ in view of
the monotonicity, it is immediate to see that it has to hold
$$
h(u_-)-h(u_+)=c(u_+-u_-),
$$
so that there exists at most a unique $c \in \mathbb{R}$ giving a monotone $\{u_-, u_+\}$--connection; we stress
that, in
the Burgers case
$h(u)=u^2/2$,
the previous condition is exactly \eqref{condC} (and in this case $v'(s) < 0$ for every $s \in \mathbb{R}$). 
Hence, here the picture is a bit different - actually, as we will see in Section 3,
the situation is much more restrictive - and we can explicitly solve the problem, due to the particular form
of the equation.
\medbreak
It is thus natural to wonder which effect on the admissible speeds prevails when combining a nonlinear advection
$h(u)$ and a nonlinear forcing term $f(u)$ in equation \eqref{QFB}, together with a (realistic) saturating and density-dependent diffusion.
This will be our main concern; precisely, we study the problem 
\begin{equation}\label{ilproblemaintro}
\left\{
\begin{array}{l}
(P[(D(v))'])' - (h'(v)+c)v' + f(v) = 0 \\
v(-\infty)=u_-, \; v(+\infty)=u_+, \quad v'(s) > 0,
\end{array}
\right.
\end{equation}
derived from the search for monotone heteroclinic traveling waves for \eqref{QFB}, and we determine the set of the admissible speeds
for different choices of $f(u)$, providing some numerical evidence, as well. Complementarily, we also start investigating the admissible
speeds in the small viscosity limit, analyzing the dependence of the critical speed on a small viscosity parameter braking the diffusion
(see Section \ref{sez6}).

In this respect, let us mention that the case without reaction was taken into account, e.g., in \cite{KurRos2}; on the other hand, the
linear
diffusion case with both a reaction and a convective term was analyzed, for instance, in \cite{MalMar02}, where it was
shown that one could in general provide only an estimate of $c^*$. 
In this framework, the nonlinear (but not saturating) 
diffusion case has been studied in \cite{MalMarMat04}, as well as in 
\cite{DePSan00, FerDePRG06} for the case of porous media-type equations.
Finally, the specific case of the linear Fisher-Burgers equation, with particular emphasis on traveling waves, has attracted the interest of several authors in these last years
(we refer, as examples, to the previously mentioned papers \cite{Mac13,ZLZ13}). 

For completeness, let us also cite the paper \cite{SanMai94}, the monography \cite{GilKer}, as well as the references \cite{Aro80,
Mur} for reviews and detailed explanation about density-dependent processes, and, as an example, the papers \cite{BonObeOma, ObeOma} for the
study of
different
boundary value problems for the mean-curvature operator. 
On the other hand,  traveling waves for the FKPP equation have been largely studied, respectively, in association
with
porous media-type diffusions or when the diffusion process takes place on a Riemannian manifold. Far from being complete in the
references, we quote,
respectively, \cite{CalCamCasSanSol15, DePSan98, StaVaz15} and \cite{MatTesPun15}, referring the reader to the references in these papers,
as
well. 
\medbreak
We close this Introduction with a plan of the paper. 
\newline
In Section \ref{sez2} we present the general procedure to reduce equation \eqref{QFB} to the
first order; we here adapt to the presence of the nonlinear term $D(u)$ the technique already exploited, for instance, in the
papers \cite{GarSan, MalMar02}. Essentially, it consists in a suitable change of variables heavily relying on the fact that we search for
\emph{monotone} solutions. For our purposes, it is here crucial that the shape of the nonlinearity $f(u)$ is invariant under
``rescalement'' by $D(u)$. Even more, this may happen in presence of a wider class of operators (see assumption $\mathbf{(P)}$ in Section \ref{sez2}). 
\newline
Section \ref{sez3} is the core of the paper and it is entirely devoted to the study of the admissible speeds for equation \eqref{QFB}, assuming the reaction term to be positive on $(0, 1)$. In particular, the main
result of the section is the following.
\begin{theorem}\label{ilprincipale}
Let $d(u)=D'(u)$ be a continuous and positive function and let $P(s)$ be as in \eqref{curvat}. Moreover, let $h: \mathbb{R} \to
\mathbb{R}$ be a $C^1$-function such that $h(0)=0$, and $f: \mathbb{R} \to
\mathbb{R}$ be a continuous function such that $f'(0)$ exists, 
$f(u) > 0$ for every $u \in \,(0, 1)\,$ and $f(u) \leq ku$, $f(u) \leq l(1-u)$ for every
$u \in [0, 1]$, for suitable constants $k, l > 0$ (notice that this implies $f(0)=0=f(1)$). Then, there
exists $c^* \in \mathbb{R}$, with
$$
2\sqrt{d(0) f'(0)} - h'(0) \leq c^* \leq 2 \sqrt{\sup_{u \in [0, 1]} \frac{d(u)f(u)}{u}} - \min_{u \in [0, 1]} h'(u),
$$
such that \eqref{ilproblemaintro} has a solution if and 
only if $c \geq c^*$.
\end{theorem}
The proof is carried out by means of a shooting technique and lower-upper solutions, after having performed the change of
variables
described in Section \ref{sez2}. We underline that Theorem \ref{ilprincipale} gives the same bound for the critical speed as in the
linear case, and finds immediate application to the
Fisher-Burgers equation. Of course, such a bound may be negative, in principle, but for sure we have an unbounded set of positive admissible
speeds. Incidentally, we observe that the analysis of
Section \ref{sez3} can
be easily extended to the case of suitable general convective terms of the form $H(u,\d_x u)$ (see, for instance, \cite{MalMar02}) .
\newline
On the other hand, in Section \ref{sez4} we start taking into account reaction terms $f(u)$ which are not positive in a right
neighborhood of zero. In Propositions \ref{possibileB} and \ref{possibileC}, we first give some sufficient conditions for the
existence
of a \emph{positive} admissible speed. Due to the presence of the
convection, we observe as well the possible appearance of \emph{a unique negative}
admissible speed  (cf. \cite{DePSan00}), differently from \cite{GarSan}; at the same time, we are able to show examples where no admissible speeds exist. The 
numerical description of these two situations is the concern of Section \ref{sez5}.
\newline
Finally, Section \ref{sez6} is devoted to the study of equation \eqref{QFB} in the small viscosity limit, meaning that we consider our 
problem with a small parameter $\e$ in front of the diffusive term. Using the results obtained in the previous sections, we show that
the positive critical speeds are proportional to $\e$.
This result is meaningful in the spirit
of the possible appearance of a phenomenon know as \emph{metastable dynamics}, as explained in details at the end of the section. In particular, we expect
the speed rate of convergence of the time dependent solutions towards their asymptotic limits to be influenced by the ``size" of the
viscosity, as already proved for the linear-diffusion case (see, for instance, the pioneering article \cite{KK86}). These features are currently the object
of our study, and may represent a possibly interesting follow-up of the present work.

\section{The change of variables}\label{sez2}

In this section, we introduce the main assumptions on the functions appearing in \eqref{QFB} and we describe the change of
variables needed in order to reduce such an equation to a first order ODE, when looking for monotone traveling waves. We will 
underline the differences
with respect to the case when $D(u)=u$, referring the reader, e.g., to the presentation given in \cite{GarSan} for the
general procedure.  
\smallbreak
We focus on the partial differential equation
\begin{equation}\label{PDEperturb}
\d_t u=\partial_x (P[\partial_x(D(u))]) - \partial_x h(u) + f(u),
\end{equation}
and we assume the following hypotheses: 
\begin{itemize}
 \item[$\mathbf{(P)}$] $P: \mathbb{R} \to \mathbb{R}$ is a $C^1$-function such that $P(0)=0$, $P'(0)=1$, $P'(s) > 0$ for every $s \in
\mathbb{R}$, and $\lim_{s \to \pm \infty} P(s) = L^\pm$, for suitable constants $L^\pm \in \mathbb{R}$; moreover,
\begin{equation}\label{singolarità}
\int_{-\infty}^{+\infty} sP'(s) \, ds < +\infty;
\end{equation}
 \item[$\mathbf{(D)}$] $D: [0, 1] \to \mathbb{R}$ is a $C^1$-function such that $D(0)=0$ and there exist
$d_0, d_1 > 0$ with
$$
0 < d_0 \leq D'(u)=d(u) < d_1, \quad \textrm{ for every } u \in [0, 1];
$$ 
 \item[$\mathbf{(H)}$] $h: [0, 1] \to \mathbb{R}$ is a $C^1$-function such that  $h(0)=0$; 
 \item[$\mathbf{(F)}$] $f: [0, 1] \to \mathbb{R}$ is a $C^1$-function such that $f(0)=0=f(1)$ and there exists $u_0 \in [0, 1)\,$ with 
$$
f(u_0)=0 \textnormal{ and } f(u) > 0  
\textnormal{ in } \,(u_0, 1)\,.
$$
\end{itemize}
A couple of comments are in order. First, hypothesis $\textbf{(P)}$ expresses the concept of saturating diffusion; the monotonicity of
$P(s)$ ensures the possibility of performing the change of variables below, while $P'(0)=1$ will be useful for the estimate
\eqref{stimalefinaleC} below, as we explain in Remark \ref{ilcurvatura}. Finally, assumption \eqref{singolarità} ensures that $R(y)$ defined in \eqref{lefunzioni} below presents a singularity, as is the case for the curvature operator. 
Concerning the last assumption, we have required for simplicity that $f \in C^1$, but many of the results we state (as we
have already seen in the statement of Theorem \ref{ilprincipale}) can be given under weaker assumptions, with minor changes. We also notice
that $\mathbf{(F)}$ states that $0$ and $1$ are equilibria (in general, the only ones) for \eqref{PDEperturb}, so that it makes sense to
search for $\{0, 1\}$--connections; in order for such connections to be \emph{increasing},
the positive sign of $f$ in a neighborhood of $1$ is necessary, as an elementary argument in the phase-plane shows. Moreover, depending on
the shape of the reaction term in a neighborhood of $0$,
we can distinguish between three different cases which
have by now become classical in the literature: 
\begin{itemize}
 \item[A)] $u_0=0$. In this case, $f(u) > 0$ for every $u \in \,(0, 1)\,$; a prototype model  is the so-called
\emph{Fisher-type} nonlinearity $f(u)=u(1-u)$. In this case, we say that $f$ \emph{is of type A} and we write $f \in \mathcal{A}$. 
 \item[B)] $u_0 > 0$ and $f(u) \equiv 0$ on $[0, u_0]$. In this case, we say that $f$ \emph{is of type B} and we write $f \in
\mathcal{B}$.
 \item[C)] $u_0 > 0$ and $f(u) < 0$ for every $u \in \,(0, u_0)\,$; a prototype model is given by the cubic nonlinearity
$f(u)=u(u-u_0)(1-u)$. In this case, we say that $f$ \emph{is of type C} and we write $f \in \mathcal{C}$. 
\end{itemize}
When looking for monotone heteroclinic traveling waves $u(t, x)=v(x+ct)$ solving \eqref{PDEperturb}, we are thus led to study the boundary
value problem
\begin{equation}\label{IIord}
\left\{
\begin{array}{l}
(P[(D(v))'])' - (c+h'(v)) v' + f(v) = 0 \\
v(-\infty)=0, \; v(+\infty)=1, \; v' > 0. 
\end{array}
\right.
\end{equation}
We observe that $D: [0, 1] \to D([0, 1])=:[0, D_1]$ is invertible (since $d(u) >
0$ implies that $D(u)$ is strictly increasing); thus, setting $w(t)=D(v(t))$, \eqref{IIord} can be rewritten as 
$$
\left\{
\begin{array}{l}
\displaystyle (P(w'))' - (c+h'(D^{-1}(w))) \frac{w'}{d(D^{-1}(w))} + f(D^{-1}(w)) = 0 \\
\\
w(-\infty)=0, \; w(+\infty)=D_1, \; w' > 0. 
\end{array}
\right.
$$
Since $D(u)$ and $u(t)$ are strictly monotone, the map $t \mapsto w(t)$ is invertible, so that we can treat
$w$ as the new independent variable and write $t=t(w)$ as a function of $w$. Setting $\phi(w)=w'(t(w))$, we can proceed
similarly
as in \cite{GarSan}, obtaining  
$$
\frac{d}{dw} Q(\phi(w)) - (c+\hat{h'}(w)) \frac{\phi(w)}{\hat{d}(w)} + \hat{f}(w) = 0,
$$
where $\hat{f}(w)=f(D^{-1}(w))$, $\hat{h'}(w)=h'(D^{-1}(w))$, $\hat{d}(w)=d(D^{-1}(w))$, and $Q(s)$ is the primitive of $sP'(s) $ satisfying $Q(0)=0$.
Writing $y(w)=Q(\phi(w))$
and denoting by $R(\cdot)$ the functional inverse of $Q$,  
\eqref{IIord} is thus equivalent to 
\begin{equation}\label{lagenerale}
\left\{
\begin{array}{l}
\displaystyle y'= \frac{(c+\hat{h'}(w))}{\hat{d}(w)} R(y) - \hat{f}(w), \\
\\
y(0)=0, \; y(D_1)=0, \; 0 < y(w) < 1 \textrm{ for } w \in \,(0, D_1)\,.
\end{array}
\right.
\end{equation}
We observe that, in view of the invertibility of $D(u)$, the shape of $\hat{f}(w)$ is preserved, namely $\hat{f}(w) \in
\mathcal{A}$ (resp. $\mathcal{B}, \mathcal{C}$) if and only if $f(u) \in \mathcal{A}$ (resp. $\mathcal{B}, \mathcal{C}$).
\\
For the sake of simplicity, from now on we consider the particular case where  
$$ 
P(v)=\frac{v}{\sqrt{1+v^2}},
$$
so that
\begin{equation}\label{lefunzioni}
Q(s)=1-\frac{1}{\sqrt{1+s^2}}, \quad 
R(y)=\frac{\sqrt{y(2-y)}}{1-y},
\end{equation}
and \eqref{lagenerale} becomes  
$$
\leqno{(\star)} \qquad \qquad \left\{
\begin{array}{l}
\displaystyle y'= \frac{(c+\hat{h'}(w))}{\hat{d}(w)} \frac{\sqrt{y(2-y)}}{1-y} - \hat{f}(w), \\
\\
y(0)=0, \; y(D_1)=0, \; 0 < y(w) < 1 \textrm{ for } w \in \,(0, D_1)\,,
\end{array}
\right.
$$
where we explicitly note that
$$
y(w)=1-\frac{1}{\sqrt{1+[d(u(t(w)))u'(t(w))]^2}}.
$$
Henceforth, problem $(\star)$ will be our main object of interest.

\section{Admissible speeds for reaction terms of type A}\label{sez3}

In this section, we study the admissible speeds for the partial
differential equation
\begin{equation}\label{NLFB}
\d_t u=\d_x(P[\d_x (D(u))] - \d_x h(u) + f(u),
\end{equation}
with
$$
P(s)=\frac{s}{\sqrt{1+s^2}},
$$
assuming that the reaction term $f(u)$ satisfies hypothesis $\mathbf{(F)}$ with $u_0 \equiv 0$, namely $f \in \mathcal{A}$.
Moreover, we require conditions
$\mathbf{(D)}$ and $\mathbf{(H)}$ to be satisfied.  
\newline
In view of the previous section, we directly focus our attention on problem $(\star)$; however, as stated in Theorem
\ref{ilprincipale}, in order
for a solution of such a problem to give rise to a ``proper'' heteroclinic traveling wave (namely, a function defined on the whole $\mathbb{R}$ reaching both equilibria in infinite time) it is
necessary to assume in addition that 
$$
\textrm{there exist } k, l > 0 \textrm{ with } f(u) \leq ku, \, f(u) \leq l(1-u) \textrm{ for every } u \in [0, 1].
$$
Otherwise, as highlighted in
\cite{BonSan}, coming back to the original variables may give rise to a ``sharp'' traveling wave, i.e., a traveling
wave reaching the equilibria in a finite time and then extended to the whole $\mathbb{R}$ with constant
value. We refer to \cite[Remarks 2.1 and 2.2]{GarSan} for further comments.
With this in mind, from now on we will always study the solvability of the first order two-point problem $(\star)$. 

\subsection{Critical speed and existence of traveling waves}

We first give a lower bound on the critical speed $c^*$. 
\begin{lemma}\label{necessariaBF}
Let $f \in \mathcal{A}$ and set $k:=f'(0)$. Let $y(w)$ be a solution to $(\star)$, with $y(w) > 0$ in a right neighborhood
of $0$.
Then,
$$
c \geq 2 \sqrt{d(0) f'(0)}- h'(0).
$$   
\end{lemma}
\begin{proof}
The proof is similar to the one of \cite[Lemma 3.1]{GarSan}.
In particular, since $y(w)$ has to be positive in a right neighborhood of $0$, we have $l=\limsup_{w \to 0^+ } (\sqrt{y(w)})' \geq 0$.
Moreover, from the differential equation in $(\star)$ we deduce 
$$
(\sqrt{y(w)})'=\frac{1}{2}\left[\frac{(c+\hat{h'}(w))}{\hat{d}(w)} \frac{\sqrt{2-y(u)}}{1-y(u)} - \frac{\hat{f}(w)}{w}
\frac{w}{\sqrt{y(w)}}\right],
$$
and, taking the upper limit for $w \to 0^+$ on both sides, we have 
$$
l = \frac{1}{2}\left[\sqrt{2}\frac{(c+\hat{h'}(0))}{\hat{d}(0)} - \frac{\hat{f}'(0)}{l}\right].
$$
Noticing that $\hat{d}(0)=d(0)$, $\hat{h'}(0)=h'(0)$ and $\hat{f}'(0)=f'(0)/d(0)$, we have 
$$
l = \frac{\sqrt{2}(c+h'(0)) \pm \sqrt{2(c+h'(0))^2-8 d(0) f'(0)}}{4d(0)}, 
$$
whence, since $l > 0$, it has to be $c+h'(0) > 0$ and 
$$
c \geq 2 \sqrt{d(0) f'(0)}- h'(0).
$$
\end{proof}
We now state and prove our main result concerning the critical speed for equation \eqref{NLFB}; 
we point out that it is a
generalization of the result previously obtained in \cite[Theorem 1.4]{MalMar02}, concerning a convection-reaction-diffusion equation
with linear diffusion (i.e., $P(s)=s$).

\begin{theorem}\label{teorema}
Let $f \in \mathcal{A}$. 
Then, there exists $c^* \in \mathbb{R}$,
with
\begin{equation}\label{stimalefinaleC}
2\sqrt{d(0) f'(0)} - h'(0) \leq c^* \leq 2 \sqrt{\sup_{u \in [0, 1]} \frac{d(u)f(u)}{u}} - \min_{u \in [0, 1]} h'(u),
\end{equation}
such that problem \eqref{IIord} has a solution if and only if $c \geq c^*$.
\end{theorem}
The main idea of the proof is to make the (unique) solution to  
\begin{equation}\label{cauchyback}
\left\{
\begin{array}{l}
\displaystyle y'= \frac{(c+\hat{h'}(w))}{\hat{d}(w)} \frac{\sqrt{y(2-y)}}{1-y}-\hat{f}(w)  \vspace{0.1 cm} \\
y(D_1)=0,
\end{array}
\right.
\end{equation}
denoted by $y_{c, f}^-(w)$, to intersect, for (possibly more than) a suitable value of $c$, with the corresponding solution to 
\begin{equation}\label{cauchyforward}
\left\{
\begin{array}{l}
\displaystyle y'= \frac{(c+\hat{h'}(w))}{\hat{d}(w)} \frac{\sqrt{y(2-y)}}{1-y}-\hat{f}(w)  \vspace{0.1 cm} \\
y(0)=0,
\end{array}
\right.
\end{equation}
which we denote by $y_{c, f}^+(w)$.
In order to prove Theorem \ref{teorema}, we need to state the following additional lemma concerning the solution to
\eqref{cauchyback}. 
\begin{lemma}\label{valoriback}
Let $f$ satisfy $(\mathbf{F})$. Then, for every 
$$
c \geq 2 \sqrt{\sup_{u \in [0, 1]} \frac{d(u)f(u)}{u}} - \min_{u \in [0, 1]} h'(u),
$$
it holds
\begin{equation}\label{sotto1}
y_{c, f}^-(w) < 1 \quad \textrm{ for every } w \in [0, D_1].
\end{equation}
Moreover,
\begin{equation}\label{positivo}
y_{c, f}^-(w) > 0 \quad \textrm{ for every }  u \in \,(D(u_0), D_1].
\end{equation}
\end{lemma}
\begin{proof}
We start proving \eqref{sotto1}. Let us fix $c \geq 2 \sqrt{\sup_{u \in [0, 1]} \frac{d(u)f(u)}{u}} - \min_{u \in [0, 1]} h'(u)$;
we notice
that, with this choice, it
holds
\begin{equation}\label{positivita}
c+\min_{u \in [0, 1]} h'(u) > 0.
\end{equation}
Assume by contradiction that there exists $\bar{w} \in [0, D_1)\,$ such that  
$$
\limsup_{w \to \bar{w}^+} y_{c, f}^-(w) = 1;
$$
since $y_{c, f}^-$ is continuous and $y_{c, f}^-(D_1)=0$, we can assume, without loss of generality, that 
$$
y_{c, f}^-(w) < 1 \quad \textrm{ for } w > \bar{w}.
$$
Then, using the differential equation in \eqref{cauchyback}, thanks to \eqref{positivita} - which ensures that the coefficient in front of
$R(y)$ therein is always positive - we are able to construct a sequence
$a_n\to\bar w$ satisfying, at the same time,  
$$
(y_{c, f}^-)'(a_n)\le 0 \quad \textnormal{and} \quad \lim_{n \to \infty} (y_{c, f}^-)'(a_n) = +\infty, 
$$
a contradiction. \newline
For the second part of the Lemma, let us now observe that $z(w) \equiv 0$ is a lower solution for \eqref{cauchyback} in $[D(u_0), 1]$, so that $y_{c, f}^-(w) \geq 0$ for every $w
\in [D(u_0), 1]$. On the other hand, if there existed $w_1 \in \,(D(u_0), 1]$ such that $y_{c, f}^-(w_1)=0$, then the differential equation
in
\eqref{cauchyback} would lead to $(y_{c, f}^-)'(w_1) < 0$, implying that $y_{c, f}^-$ is decreasing in a neighborhood of $w_1$, a
contradiction. Hence, \eqref{positivo} is proved. 
\end{proof}
\begin{proof}[\bf Proof of Theorem \ref{teorema}]
Lemma \ref{necessariaBF} ensures the first inequality in \eqref{stimalefinaleC}; our aim is now to show that every $c \in \mathbb{R}$ such
that
\begin{equation}\label{altrastima}
c \geq 2 \sqrt{\sup_{u \in [0, 1]} \frac{d(u)f(u)}{u}} - \min_{u \in [0, 1]} h'(u)
\end{equation}
is admissible.
In view of Lemma \ref{valoriback}, the solution $y_{c, f}^-(w)$ is defined for every $w \in [0, D_1]$. 
On the other hand, fixed $c$ satisfying \eqref{altrastima} we want to construct a positive lower solution to \eqref{cauchyforward}. To
this end, we  use the solutions to 
\begin{equation}\label{cauchyaux}
\left\{
\begin{array}{l}
\displaystyle y'= \frac{\beta }{d(D^{-1}(w))} \frac{\sqrt{y(2-y)}}{1-y}   \\
\\
y(0) = 0,
\end{array}
\right.
\end{equation}
in dependence of $\beta > 0$. 
Indeed, noticing that $\int 1/d(D^{-1}(w)) \, dw = D^{-1}(w)$,
problem \eqref{cauchyaux} has the solution 
$$
z(w)=1-\sqrt{1-\beta^2 D^{-1}(w)^2},
$$
that is defined in a suitable right neighborhood $[0, \sigma)\,$ of $w=0$, for some $\sigma >0$. We want to choose $\beta$ in such a way that $z(w)$
solves the differential
inequality
$$
z' \leq (c+\hat{h}'(w))R(z) - \hat{f}(w)
$$
in a right neighborhood of $w=0$. Since
$$
R(z(w))  =\frac{\beta D^{-1}(w)}{\sqrt{1-\beta^2 D^{-1}(w)^2}},
$$
we thus have to choose $\beta > 0$ so as to guarantee that, for every $w \in [0, D_1]$, it holds
$$
\frac{\beta^2 D^{-1}(w)}{d(D^{-1}(w))\sqrt{1-\beta^2 D^{-1}(w)^2}} - \frac{\beta (c+ h'(D^{-1}(w)))D^{-1}(w)}{d(D^{-1}(w))\sqrt{1-\beta^2
D^{-1}(w)^2}} + f(D^{-1}(w)) \leq 0 ,
$$
whence, for every $u \in [0, 1]$,  
\begin{equation}\label{stimabeta}
\beta^2 u - \beta (c+ h'(u))u + u \frac{f(u) d(u)}{u} \sqrt{1-\beta^2 u^2} \leq 0.
\end{equation}
Dividing by $u > 0$, taking the maximum for $u \in [0, 1]$ and solving for $\beta$, 
we end up with
$$
\beta=\frac{1}{2} \left[c+\min_{u \in [0, 1]} h'(u) \pm \sqrt{(c+\min_{u \in [0, 1]}h'(u))^2 - 4\sup_{u \in [0, 1]} \frac{d(u)
f(u)}{u}}\;\right],
$$
which yields a positive value of $\beta$ thanks to \eqref{altrastima} (which, as already observed, implies that $c+\min_{u \in [0, 1]} h'(u)
>
0$).
Thus, for this choice of $\beta$, $z(w)$ is a positive lower solution to \eqref{cauchyforward} 
in the interval $[0, \sigma)$. As a consequence, there exists a solution $y_{c, f}^+(w)$ to \eqref{cauchyforward} satisfying  
\begin{equation}\label{positiva}
y_{c, f}^+(w) \geq z(w) > 0 \quad \textrm{ for every } w \in \,(0, \sigma)\,.
\end{equation}
The argument is now a standard uniqueness one (see \cite{GarSan} for further details): by continuity and
connectedness (in view of the previous considerations about $y_{c, f}^-(w)$), the two solutions $y_{c,
f}^+(w)$ and $y_{c, f}^-(w)$ have to intersect, hence they have to be the same. Thus, we have found 
the desired solution to problem $(\star)$.
\end{proof}
\begin{remark}\label{ilcurvatura}
\textnormal{As already anticipated, our explicit choice of $P(s)$ is mainly due to the importance of the curvature operator in literature and to
the need for a certain readability of the main statement; however, it can be seen that hypothesis $\textbf{(P)}$ (together with
the consequent properties of $R(y)$) is sufficient in order to
perform the above arguments. Indeed, $P'(0)=1$ guarantees, through l'H\^opital rule, that $R(y) \thicksim \sqrt{2y}$
for $y \to 0$, so that Lemma \ref{necessariaBF} holds true also under assumption $\textbf{(P)}$ (otherwise, using the same reasonings we
would obtain a different bound on $c^*$); on the other hand, \eqref{singolarità} yields a singularity for $R(y)$ (which in the case \eqref{curvat} is obtained for $y=1$).} 
\end{remark}
We now examine some statements which can be obtained through the first-order approach - possibly directly from Theorem \ref{teorema} - for
some particular cases of equation \eqref{NLFB}, among which, the Fisher and the Burgers ones.

\subsection{Quasilinear Fisher-type equations}

Setting $h(u) \equiv k \in \mathbb{R}$, and assuming $f \in \mathcal{A}$ and $P(s)$ as in \eqref{curvat}, equation \eqref{QFB}  turns into
\begin{equation}\label{PDEFisher}
\d_t u=\d_x\left(\frac{\d_x(D(u))}{\sqrt{1+\d_x(D(u))^2}}\right) + f(u).
\end{equation}
In the case $D(u) = u$, the admissible speeds for \eqref{PDEFisher} were determined in
\cite[Proposition 3.2]{GarSan}. We can now obtain such a result
as a consequence of the following corollary. 
\begin{corollary}\label{ilcorollario}
Let $f \in \mathcal{A}$. Then, there exists $c^* > 0$,
with
$$
2\sqrt{d(0) f'(0)} \leq c^* \leq 2 \sqrt{\sup_{u \in [0, 1]} \frac{d(u)f(u)}{u}},
$$
such that \eqref{PDEFisher} has a heteroclinic traveling wave connecting monotonically $0$ and $1$ if and only if $c \geq c^*$.
\end{corollary}
Indeed, whenever $f(u) \leq f'(0)u$ on $[0, 1]$, the case $d(u) \equiv 1$ yields in turns $c^*=2 \sqrt{f'(0)}$, as
in
\cite{GarSan}; 
actually, an estimate of $c^*$ was therein given under the weaker bound 
$$
f(u) \leq \frac{Mu}{\sqrt{1-\min\{M, 1\}u^2}},
$$
in terms of $\sqrt{M}$. However, this is the case also for us, by slightly
modifying the end of the proof of Theorem
\ref{teorema}. The nonlinear Fisher equation, obtained by choosing $f(u)=u(1-u)$, fits into this framework.

\subsection{Quasilinear Burgers-type equations}

With reference to equation \eqref{QFB}, we now consider the case when $P(s)$ is given by \eqref{curvat} and $f(u) \equiv 0$, and we
analyze the associated first-order model. Explicitly, we thus study the viscous conservation law
\begin{equation}\label{PDEB}
\d_t u=\d_x\left(\frac{\d_x(D(u))}{\sqrt{1+\d_x(D(u))^2}}\right)  - \d_x h(u),
\end{equation}
which reduces to the Burgers equation with mean-curvature type diffusion when choosing $h(u)=u^2/2$. By proceeding with the change of variables 
described in Section \ref{sez2}, we
 have to solve the first order two-point problem
\begin{equation}\label{BIord}
\left\{
\begin{array}{l}
\displaystyle y'= \frac{c+\hat{h'}(w)}{\hat{d}(w)} \frac{\sqrt{y(2-y)}}{1-y},  \vspace{0.1 cm} \\
y(0)=0=y(D_1),\; y(w) > 0 \textnormal{ on } (0, D_1).
\end{array}
\right.
\end{equation}
We observe that, if $c < -h'(0)$, then the solution shot forward from $0$ is identically zero; hence, we restrict our attention to the speeds $c$ such that $c \geq -h'(0)$ (observe that this bound is also given by Lemma
\ref{necessariaBF}). \newline
It is now immediate to see that Theorem \ref{teorema} cannot be applied to this equation: 
since assumption $(\mathbf{F})$ is not satisfied, Lemma \ref{valoriback} fails. The difficulty lies in the fact that it is not ensured that both the forward solution $y_{c, f}^+$ and the backward one $y_{c, f}^-$ are both positive and everywhere defined. 
Thus, an intersection between $y_{c, f}^+(w)$ and $y_{c, f}^-(w)$ may not appear anymore.  
Finally, in this case we may also lose uniqueness from $D_1$. 
\newline
Hence, at first, we need to understand the behavior of the solutions shot backward
from $D_1$: denoted by $y(w)$ one of these solutions, the particular structure of the differential equation in \eqref{BIord} allows us
to directly integrate by
separation of variables, obtaining
$$
-\sqrt{y(w)(2-y(w))}=c-cD^{-1}(w)+h(1)-h(D^{-1}(w)). 
$$
Thus, the right-hand side in such an equality has to be negative for $w \in (0, D_1)$, so that it has to be
\begin{equation}\label{acca1}
c-cu+h(1)-h(u) < 0, \quad \textrm{ for every } u \in (0, 1).
\end{equation}
Since it is required that $0 < y(w) < 1$ for every $w \in (0, D_1)$, it now follows that
$$
y_c^-(w)=1-\sqrt{1-(cD^{-1}(w)-c+h(D^{-1}(w))-h(1))^2}, \quad \textrm{ for every } w \in (0, D_1),
$$
and
\begin{equation}\label{acca2}
cu-c+h(u)-h(1) < 1, \quad \textrm{ for every } u \in (0, 1).
\end{equation}
If \eqref{acca1} and \eqref{acca2} hold, we now impose $y_{c}^-(0)=0$ and obtain a unique admissible value of $c$ for the existence of an increasing \{0, 1\}-connection:
$$
c=-h(1). 
$$
This is compatible with \eqref{acca1} and \eqref{acca2} only if 
\begin{equation}\label{accafinale}
h(u)-1 < h(1) u < h(u) \quad \textrm{ for every } u \in (0, 1).
\end{equation}
Such a condition imposes a concavity shape on $h(u)$, and is indeed not fulfilled by the Burgers flux $h(u)=u^2/2$, as is expected in view of the portrait for the Burgers equation highlighted in the Introduction. On the other hand, a concave flux like $h(u)=-u^2$ enters this setting and we observe that in this case $c^*=1$.   
Of course, this argument works whatever the limits the heteroclinic solution has to approach (we recall
that, for equation \eqref{PDEB}, all the constants are equilibria);
in particular, fixed two equilibria $u_-$ and $u_+$ to be connected, direct
computations show that, under the validity of conditions analogous to \eqref{acca1} and \eqref{acca2}, it has to be
\begin{equation}\label{velocitaTV}
c=\frac{ h(u_-)-h(u_+)}{u_+-u_-}.
\end{equation}
On the other hand, according to \cite[Remark 2.3]{GarSan}, \emph{decreasing} \{0, 1\}-connections for \eqref{PDEB} can be obtained by solving  
$$
\left\{
\begin{array}{l}
\displaystyle y'= -\frac{c+\hat{h'}(w)}{\hat{d}(w)} \frac{\sqrt{y(2-y)}}{1-y},  \vspace{0.1 cm} \\
y(0)=0=y(D_1),\; y(w) > 0 \textnormal{ on } (0, D_1),
\end{array}
\right.
$$
and a similar argument can be performed, showing that a decreasing connection exists, taking $c$ as in \eqref{velocitaTV}, under the validity of 
\begin{equation}\label{accafinale2}
h(u) < h(1) u < h(u)+1 \quad \textrm{ for every } u \in (0, 1).
\end{equation}
We can thus state the following proposition.
\begin{proposition}\label{PropB}
Let $h(u)$ satisfy either condition \eqref{accafinale} or \eqref{accafinale2}. Then, a monotone 
(increasing or decreasing, respectively) 
$\{0, 1\}$--connection of traveling wave type $u(t, x)=v(x+ct)$ solving \eqref{PDEB} exists
if and only if $c=-h(1)$. 
\end{proposition}
Such a proposition can be easily adapted to $\{u_-, u_+\}$-connections (with $u_-$, $u_+$ arbitrary distinct equilibria), suitably changing conditions \eqref{accafinale} and \eqref{accafinale2}, and can be applied, for instance, to the Burgers case 
$h(u)=u^2/2$ recalled in the Introduction, recovering the mentioned picture.

\subsection{Quasilinear Fisher-Burgers type equation}

To make our discussion complete, we here consider both the effects of a reaction and a convection term,
choosing for simplicity $P(s)$ as in \eqref{curvat}, $h(u)= \kappa u^2$, for some $\kappa >0$, and $f \in \mathcal{A}$ such
that $f(u) \leq f'(0)\, u$ for all $u \in [0,1]$ (the quasilinear Fisher-Burgers equation is recovered for the choices
$\kappa=1/2$ and $f(u)=u(1-u)$).
\newline
From \eqref{stimalefinaleC}, it follows immediately that it holds
 \begin{equation*}
2\sqrt{d(0) f'(0)} \leq c^* \leq 2 \sqrt{d_1 f'(0)}.
\end{equation*}
In particular, in the Fisher-Burgers case with $d(u) \equiv 1$, it holds  $f'(0)=1$ and the critical speed $c^*$ is exactly given by
$c^* = 2$.
On the other hand, if we apply formula \eqref{stimalefinaleC} in the case $h(u)=\kappa u^2$ for $\kappa < 0$, \eqref{stimalefinaleC}
 only leads, in general, to an estimate of $c^*$. Explicitly, we have 
$$
2\sqrt{d(0) \,f'(0)} \leq c^* \leq 2\sqrt{d_1
\,f'(0)}-2\kappa
$$ 
and it appears more complicated, in view of the nonlinear diffusion, to obtain the precise value of $c^*$ as in \cite[Section
5]{MalMar02}. However,
the case of convex fluxes (corresponding to the choice $\kappa > 0$) appears to be more significant in literature. 

\section{Admissible positive speeds for reaction terms of type B and C} \label{sez4}

We are now interested in studying the admissible speeds for problem $(\star)$ when $f$ may
vanish inside the interval $(0, 1)$. We will see that, in general, the presence of a convective term 
changes the picture with respect to the case $h(u) \equiv 0$, and the situation is more
complicated than in the previous section: even if $\int_0^1 f(u) \, du > 0$, the occurrence
of a unique {\it negative} admissible speed (compare, for instance, with \cite{DePSan00}), as well as the nonexistence of admissible speeds, are here possible. We refer the reader to Section \ref{sez5} for more details. 
\newline
To begin with, in this section we
thus first deal with the problem of the existence of {\it positive} admissible speeds. 
One possible motivation for the
search for positive speeds, among the others, is that
they possess some interesting features in dependence on a small parameter braking the diffusion, as we will see in Section \ref{sez6}. Moreover, there are cases where we have no chance for the existence of negative speeds: for instance, if $f \in
\mathcal B$ and $h'(0)=0$ (as in the Burgers case), then $c < 0$ implies that the only solution to the differential equation in $(\star)$ with $y(0)=0$ is the zero one.
\smallbreak
As in Section \ref{sez3}, we restrict ourselves to the case when $P(s)$ is of mean curvature type, defined in \eqref{curvat}. Moreover, we
assume the following condition on the convective term $h$:
\begin{itemize}
 \item[$\mathbf{(H')}$]  $h: [0, 1] \to \mathbb{R}$ is a $C^1$-function such that $h(0)=0$ and $h'(u) > 0$ for every $u \in (0,
1)$.
\end{itemize}
This is the case, for instance, of a convex convective term like the one appearing in the classical Burgers equation. \newline 
We first state a preliminary lemma regarding again some monotonicity properties of the solutions
with respect to $c > 0$. As before, we denote by $y_{c, f}^-$ the solution
to \eqref{cauchyback}, which is unique thanks to the fact that $R(y)$ is increasing. On the other hand, this time we also have uniqueness
(and continuous dependence on $c > 0$) of the
positive solution to \eqref{cauchyforward}, which will be denoted by $y_{c, f}^+$; this follows from direct integration if $f \in
\mathcal{B}$, and 
from \cite[Theorem 3.9]{BonSan} if $f \in \mathcal{C}$. 
\begin{lemma}\label{lemma1}
Let $f \in \mathcal{B}$ or $f \in \mathcal{C}$ and let $0<c_1<c_2$. Then, the following inequalities hold:
\begin{equation*}
\begin{aligned}
&y^+_{c_2,f}(w) \geq y^+_{c_1,f}(w), \quad &\forall \ w \in (0,D_1) \ \ &{\rm s.t.} \ \ y^+_{c_2,f}(w) <1, \\
&y^+_{c_2,f}(w) > y^+_{c_1,f}(w), \quad &\forall \ w \in (0,D(u_0)) \ \ &{\rm s.t.} \ \ y^+_{c_2,f}(w) <1,
\end{aligned}
\end{equation*}
being $y^+_{c,f}(w)$ the unique (positive) solution to \eqref{cauchyforward}. Similarly,
\begin{equation*}
\begin{aligned}
&y^-_{c_2,f}(w) \leq y^-_{c_1,f}(w), \quad &\forall \ w \in (0,D_1), \ \  & \ \ \\
&y^-_{c_2,f}(w) < y^-_{c_1,f}(w), \quad &\forall \ w \in (0, D_1) \ \ &{\rm s.t.} \ \ y^-_{c_2,f}(w) > 0,
\end{aligned}
\end{equation*}
being $y^-_{c,f}(w)$ the unique solution to \eqref{cauchyback}. Moreover, if $y^-_{c_2,f}(D(u_0)) > 0$, then 
\begin{equation*}
y^+_{c_2,f}(D(u_0)) >y^+_{c_1,f}(D(u_0)) \quad {\rm and} \quad y^-_{c_2,f}(D(u_0)) < y^-_{c_1,f}(D(u_0)).
\end{equation*}
\end{lemma}

\begin{proof}
In view of the uniqueness, the statement follows for instance from \cite[Exercise 4.1]{Har}. 
\end{proof}

As a direct consequence of Lemma \ref{lemma1}, in the case $f \in \mathcal B$ or $f \in \mathcal C$, if a positive admissible speed exists then it is necessarily unique. In the figures in the next sections, 
we will depict the connection corresponding to such a unique speed in some different cases (see Figures \ref{figB1}, \ref{figB2}, \ref{figC1} and \ref{figC2}). In this respect, we explicitly mention that we will always show the solution to the first order problem ($\star$) under study; indeed, a direct numerical search for the heteroclinic solution - that is the same, for the admissible speed - for problem \eqref{IIord} is much more difficult, since the boundary conditions $v(-\infty)=0$, $v(+\infty)=1$  require a too high degree of subjectivity in choosing the initial data for the numerical simulations. We thus see another advantage of dealing with the equivalent first-order two-point problem on [0, 1]. We also observe that the simulations have been realized through the use of a numerical
Euler method (using the software Mathematica$^\copyright$), with a sufficiently small step (0.0005) and an accuracy of at least $10^{-6}$, so that the situations depicted appear
sufficiently reliable.

\subsection{Existence of a positive critical speed for $f \in \mathcal{B}$}

We recall that $f \in \mathcal{B}$ means that there exists $u_0 \in (0, 1)$ such that $f(u)=0$ for every $u \in [0, u_0]$ and
$f(u)>0$ in $(u_0, 1)$. To start with, we observe that the unique positive solution to \eqref{cauchyforward} for $c=0$, denoted as before
by $y^+_{0,f}(w)$, can be obtained by
direct integration in the interval $[0, D(u_0)]$, leading to 
\begin{equation}\label{inizio0}
y^+_{0,f}(w)= 1- \sqrt{1-(h(D^{-1}(w))-h(0))^2}.
\end{equation}
In particular, in view of assumption $\mathbf{(H')}$ we immediately see that $y^+_{0,f}(D(u_0))$ exists and $y_{0, f}^+(D(u_0)) < 1$ if and only if $h$ satisfies
\begin{equation}\label{stimah}
 h(u_0) < 1+h(0). 
\end{equation}
Thus, this turns out to be a necessary condition for the existence of a positive admissible speed (otherwise, the solution shot from the
left blows-up). Clearly, such a condition is anyway not sufficient: by the monotonicity properties previously highlighted, the existence of
a positive admissible speed will indeed be possible only if 
\begin{equation}\label{quipos}
y_{0, f}^+(D(u_0)) < y_{0, f}^-(D(u_0)),
\end{equation}
so that some estimate on the backward solution $y_{0, f}^-$ is required, as well. 
\newline
With the following proposition, we aim at showing that it is possible to find explicit situations when \eqref{quipos} is fulfilled,
being aware of the nonoptimality of the statement; actually, a precise estimate of the solutions of the differential equation in ($\star$)
is far from being straight. 
In order to simplify the statement, we here assume that $D(u) = u$, so that $d(u) \equiv 1$; the extension of the statement to density-dependent diffusions is straight, as will be clear from the proof. 
\begin{proposition}\label{possibileB}
Let $f \in \mathcal{B}$ and let $h(u)$ fulfill \eqref{stimah}. Fix $\eta > 0$ such that  
$$
\eta < \min\{F^+, 1\},
$$
where
$F^+=\int_{u_0}^1 f(u) \, du$.
Then, sufficient conditions for the existence of a \emph{positive}
admissible speed for problem $(\star)$ are the following:
\begin{itemize}
 \item[1)] if $F^+ < 1$, 
\begin{equation}\label{accaB1}
\left\{
\begin{array}{l}
h(u_0)-h(0) < \sqrt{1-(1+\eta-F^+)^2} \vspace{0.2 cm} \\ 
\displaystyle h(1)-h(u_0) < \frac{(1-F^+)}{\sqrt{F^+(2-F^+)}}\, \eta \,;  \vspace{0.2 cm} 
\end{array}
\right.
\end{equation}
 \item[2)] if $F^+ \geq 1$, fixed two other constants $\kappa$ with $0 < \kappa < \eta$ and  $\xi_0 \in (u_0, 1)$
such that $\int_{\xi_0}^1 f(u) \, du = 1-\kappa$, 
\begin{equation}\label{accaB2}
\left\{
\begin{array}{l}
\displaystyle h(u_0)-h(0) < \sqrt{1-\eta^2} \vspace{0.2 cm} \\ 
\displaystyle h(\xi_0)-h(u_0) < \sqrt{1-\Big(\frac{\kappa+\eta}{2}\Big)^2}-\sqrt{1-\eta^2} \vspace{0.2 cm} \\
\displaystyle h(1)-h(\xi_0) < \frac{\eta-\kappa}{2} \frac{\kappa}{\sqrt{1-\kappa^2}} \,.
\end{array}
\right.
\end{equation}
\end{itemize}
\end{proposition}
\begin{proof}
We preliminarily recall that the unique positive solution $y_{0, f}^+$ to \eqref{cauchyforward} for $c=0$ is given by \begin{equation}\label{inizio}
y_{0, f}^+(u) = 1- \sqrt{1-(h(u)-h(0))^2}
\end{equation}
(recall that $d(u) \equiv 1$). 
We first consider case 1). In view of 
$\eqref{accaB1}_i$, formula \eqref{inizio} immediately gives that 
$
y_{0, f}^+(u_0) < F^+ -\eta.
$
On the other hand, $\bar{y}(u)=\int_u^1 f(s) \,
ds$ is a supersolution for
\eqref{cauchyback} which always stays bounded below from $1$. Consequently, in view of the monotonicity of $R(y)$, we have
\begin{eqnarray*}
y_{0, f}^-(u_0) && = \int_{u_0}^1 f(u) \, du - \int_{u_0}^1 h'(u) \frac{\sqrt{(y_{0, f}^-(u))(2-y_{0, f}^-(u))}}{1-y_{0, f}^-(u)} \, du \\
&& > F^+ - (h(1)-h(u_0)) \frac{\sqrt{F^+(2-F^+)}}{1-F^+},
\end{eqnarray*}
where $y_{0, f}^-$ denotes the solution to \eqref{cauchyback} with $c=0$. 
In view of $\eqref{accaB1}_{ii}$, this yields 
$
y_{0, f}^-(u_0) > F^+ -\eta, 
$
so that \eqref{quipos} is fulfilled and the conclusion follows from the usual monotonicity arguments.
\smallbreak
\noindent
On the contrary, if we are in case 2), we use $\eqref{accaB2}_i$ to infer, through formula \eqref{inizio}, that 
$
y_{0, f}^+(u_0) < 1-\eta. 
$
Reasoning similarly as in case 1) in the interval $(\xi_0, 1)$, we also have 
$$
y_{0, f}^-(\xi_0) > 1-\kappa - (h(1)-h(\xi_0)) \frac{\sqrt{1-\kappa^2}}{\kappa},
$$
so that, under the validity of $\eqref{accaB2}_{iii}$, we deduce  
$$
y_{0, f}^-(\xi_0) > 1-\frac{\kappa+\eta}{2}.
$$ 
Now we can apply the same argument as in \cite[Proposition 3.9, (3), Case 2]{GarSan}: in particular, the solution to 
$$
z'=h'(u) \frac{\sqrt{z(2-z)}}{1-z}, \qquad z(\xi_0)=1-\frac{\kappa+\eta}{2}
$$
is a subsolution for $y_{0, f}^-(u)$ in the interval $[u_0, \xi_0]$. 
Since, for every $u \in (u_0,\xi_0)$, 
$$
z(u)=1-\sqrt{1-(h(u)-h(\xi_0)+\sqrt{1-(\kappa+\eta)^2/4})^2},
$$
in view of $\eqref{accaB2}_{ii}$ we obtain that $z(u_0) > 1-\eta$, so that 
$
y_{0, f}^+(u_0) > 1-\eta
$ 
and the conclusion follows as before.
\end{proof}
In the previous statement, the constants $\eta$ and $\kappa$ are meant to control the behavior of the solutions $y_{c, f}^+$ and $y_{c,
f}^-$ and may be chosen quite freely. However, as expectable, it is immediate to see that the
controls $\eqref{accaB1}_i$ and $\eqref{accaB2}_i$ are more restrictive as
long as $\eta$ grows and, symmetrically, the more $\kappa$ approaches $\eta$, the stronger the controls $\eqref{accaB2}_{ii}$ and
$\eqref{accaB2}_{iii}$ become, so that a
certain care is required. The only substantial change in presence of a density-dependent diffusion is represented by the appearance of $d(u)$ under the integral involving $f(u)$, so that the above existence conditions should be slightly modified in order to obtain the same statement. 
\begin{remark}
\textnormal{When no convection effects are present in the considered model, namely  $h'(u) \equiv 0$, assumptions \eqref{accaB1} and 
\eqref{accaB2}  are trivially satisfied, and Proposition \ref{possibileB} 
may thus be adapted so to represent the natural generalization of \cite[Proposition 3.9]{GarSan} to the case of density-dependent diffusions.}
\end{remark}

We conclude this subsection with some numerical simulations showing an application of Proposition \ref{possibileB} when $F^+$ is either
less or greater than one (Figures \ref{figB1} and \ref{figB2} respectively). Also, in Figure \ref{figB3} we show  the necessity of condition
\eqref{stimah} for the existence of increasing $\{0,1\}$-connections.

\begin{figure}[!h]
\centering
\includegraphics[width=11cm,height=6cm]{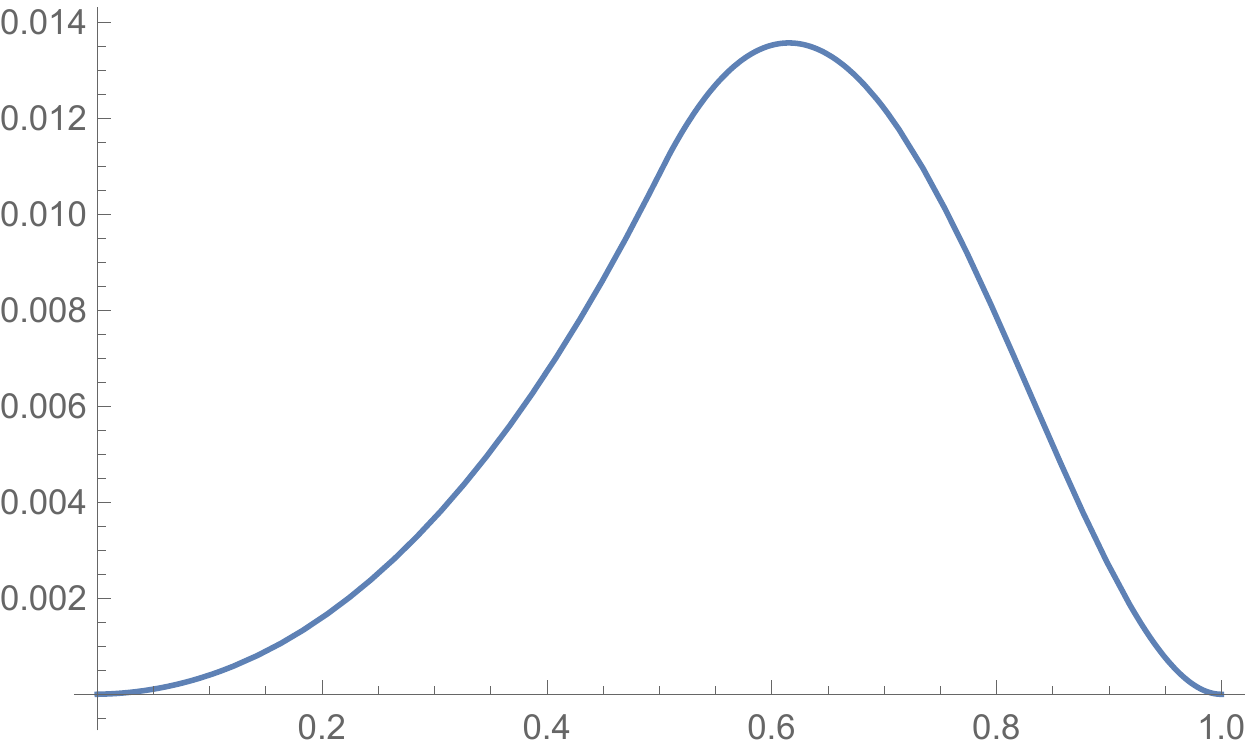}
\caption{\small{As an application of Proposition \ref{possibileB}, we consider the first-order system
$(\star)$ for $h(u)=0.05 u^2$ and $f(u)=2u(1-u)(\max\{0.5, u\}-0.5)$, and we show numerically that a positive admissible speed exists.
Indeed, in this case $F^+=1/32$ and we can take $\eta=1/64$, so that $h(0.5)-h(0)=0.0125$ and $h(1)-h(0.5)=0.0375$ enter condition
\eqref{accaB1}. We find $c^* \approx 0.2663$.}}
\label{figB1}
\end{figure}
\begin{figure}[!h]
\centering
\includegraphics[width=11cm,height=6.7cm]{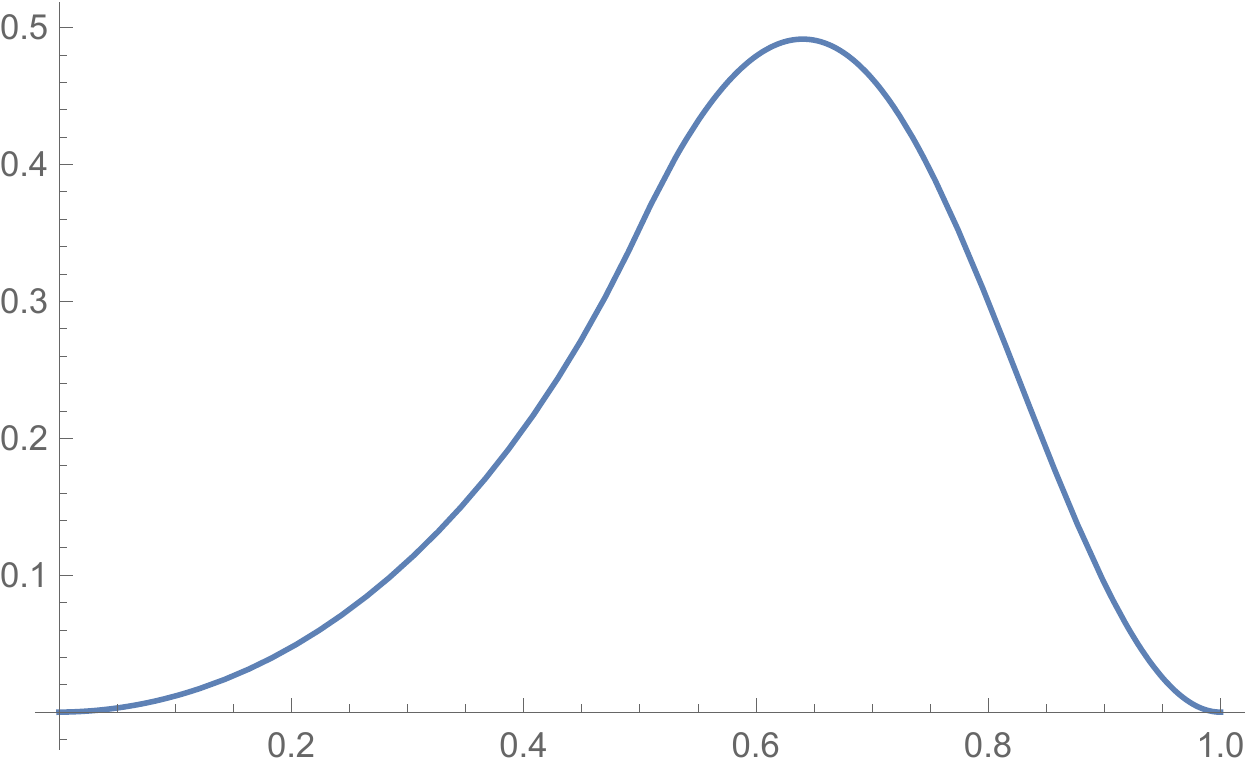}
\caption{\small{Here $h(u)=0.05 u^2$ and the reaction term is given by $f(u)=0$ for $u \in [0, 0.5]$ and $f(u)=1.5 \sin(2\pi x - \pi)(1+2x)$
for  $u \in (0.5, 1)$. It follows that $F^+ \approx 1.19366$ and $\xi_0 \approx 0.676$. Choosing $\eta=0.85$
and $\kappa=0.1$, we find, respectively, the bounds $0.52678, 0.35332, 0.037689$ in \eqref{accaB2}, which are all fulfilled by $h(u)$. We
find $c^* \approx 1.4952$. As is expectable, the estimates for $y_{0, f}^-$, performed in two steps, 
are quite conservative and the result appears less refined than in the case $F^+ < 1$.}}
\label{figB2}
\end{figure}

\begin{figure}[!h]
\centering
\includegraphics[width=11cm,height=6.7cm]{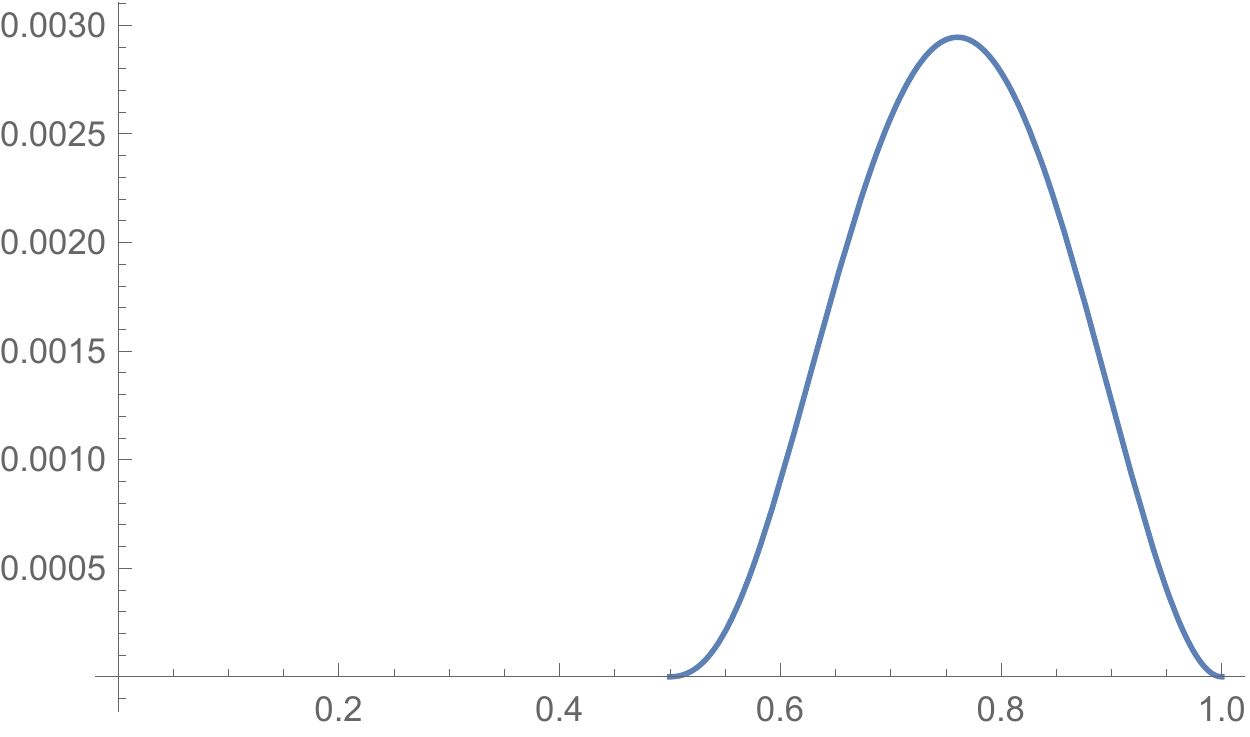}
\caption{\small{In this case, $c=0$, $f(u)$ is as in Figure \ref{figB1}, while $h(u)=\log (u+0.05)$ does not fulfill
the necessary condition \eqref{stimah}. Consequently, we do not have
monotone \{0, 1\}--connections with positive speed. However, since $f$ is of type $A$ when restricted to the interval $[0.5, 1]$, we can apply Theorem
\ref{teorema} and see that a \{0.5-1\}-connection appears already for $c=0$ (indeed,
\eqref{stimalefinaleC} provides here a negative
value of $c^*$, as it is easy to check).}}
\label{figB3}
\end{figure}

\subsection{Existence of a positive critical speed for $f \in \mathcal{C}$}

We now consider $f \in \mathcal{C}$, namely there exists $u_0 \in (0, 1)$ for which $f(u) < 0$ for $u \in (0, u_0)$ and $f(u) > 0$ for $u
\in (u_0, 1)$. As a notation, we set $f^+=\max\{f, 0\}$, $f^-=\max\{-f, 0\}$ and
\begin{equation}\label{notintegrale}
F^+ = \int_0^1 f^+(u) \, du, \qquad F^-= \int_0^1 f^-(u) \, du. 
\end{equation}
In this case, it is immediate to see that the necessary condition \eqref{stimah} may not be sufficient to
ensure that the forward solution $y_{0, f}^+$ to \eqref{cauchyforward} is such that $y_{0, f}^+(D(u_0))$ exists and is bounded away from $1$, since the
function defined in \eqref{inizio} is here only a subsolution. The existence of an admissible speed (either positive or negative) is thus
even more
influenced, in this case, by the interplay between the growths of $h'$ and $f^-$.
\smallbreak
As before, we are going to examine some conditions related to the existence of positive admissible speeds, in the easier case $D(u) = u$ (similarly as before, slight changes allow to cover also the general situation). First, a
necessary condition can be
deduced from the case $h \equiv 0$ (cf. \cite[Proposition 3.9]{GarSan}).
\begin{proposition}\label{ftipoC}
Let $f \in \mathcal{C}$ and let $h$ fulfill $(\mathbf{H}')$. Then, if $c > 0$ is an admissible speed for 
$(\star)$, necessarily it has to be  
\begin{equation}\label{necessf}
\int_0^1 f(u) \, du >0, \qquad {\rm and} \qquad
\int_0^1 f^-(u) \, du <1,
\end{equation}
and \eqref{stimah} has to hold. 
\end{proposition}

\begin{proof}
We just sketch the argument (which is similar to the one in \cite[Proposition 3.9]{GarSan}). On one hand, if
$\eqref{necessf}_{i}$ is violated, it follows that $\int_0^1
f^+(u) \, du \leq \int_0^1 f^-(u) \, du$, namely
\begin{equation*}
y^+_{0,f} (u_0) \geq- \int_0^{u_0} f(u) \, du \geq y^-_{0,f} (u_0),
\end{equation*}
and the nonexistence of a positive admissible speed follows from the usual monotonicity argument (see Lemma \ref{lemma1}). 
\\
On the other hand, if  $\eqref{necessf}_{ii}$ is not true, we have that $z(u)=- \int_0^u f(s) \, ds$ is a lower solution from the left for 
\begin{equation*}
y'= (c+h'(u)) \frac{\sqrt{y(2-y)}}{1-y}-f(u)  \geq -f(u), \quad y(0)=0
\end{equation*}
(recall that $h'(u) > 0$).
Thus, for $c > 0$ we have that $y^+_{c,f}(u) > z(u)$ as long as $z(u)<1$, meaning that $y^+_{c,f}(u)$ reaches the value $1$ in finite time. Hence, no positive
admissible speeds exists.
\newline
Finally, the necessity of the condition on $h'$ comes from the above discussion, comparing with the case $f \in \mathcal{B}$.
\end{proof}

We explicitly notice that the necessary condition \eqref{necessf}, together with \eqref{stimah}, is not
sufficient for the existence of a positive admissible speed, as it instead was for $h' \equiv 0$ (compare with \cite{GarSan}). Actually, the forward solution
$y_{0, f}^+(u)$ may blow-up even if \eqref{stimah}
and \eqref{necessf} are satisfied. 
However, if $h'(u)$ has a very slow growth, we can recover the sufficiency of \eqref{necessf} without further
assumptions, as shown in the following (nonoptimal) statement. 

\begin{proposition}\label{possibileC}
Let $h$ fulfill \eqref{stimah} and $f \in \mathcal{C}$ satisfy \eqref{necessf}. Fix $\eta$ such that  
$$
0 < \eta < \min\{F^+, 1\}-F^-.
$$
Then, sufficient conditions for the existence of a \emph{positive}
admissible speed for problem $(\star)$ (with $d(u) \equiv 1$) are the following:
\begin{equation}\label{accaC1}
h(u_0)-h(0) < \displaystyle \frac{(1-F^- - \eta)}{\sqrt{(F^- + \eta)(2-F^- -\eta)}}\, \eta 
\end{equation}
and
\begin{itemize}
 \item[1)] if $F^+ < 1$, 
\begin{equation}\label{accaC2}
\displaystyle h(1)-h(u_0) < \frac{(1-F^+)}{\sqrt{F^+(2-F^+)}}\, (F^+-F^- -\eta) \,;  \vspace{0.2 cm} 
\end{equation}
 \item[2)] if $F^+ \geq 1$, fixed other constants  $0 < \zeta < 1-F^--\eta$ and  $\xi_0 \in (u_0, 1)$
such that $\int_{\xi_0}^1 f(u) \, du = 1-\zeta$, 
\begin{equation}\label{accaC3}
\left\{
\begin{array}{l}
\displaystyle h(\xi_0)-h(u_0) < \sqrt{1-\Big(\frac{1+\zeta-F^--\eta}{2}\Big)^2}-\sqrt{1-(1-F^- -\eta)^2} \vspace{0.2 cm} \\
\displaystyle h(1)-h(\xi_0) < \frac{1-\zeta-F^--\eta}{2} \frac{\zeta}{\sqrt{1-\zeta^2}} \,.
\end{array}
\right.
\end{equation}
\end{itemize}
\end{proposition}
\begin{proof}
The idea is to proceed similarly as in Proposition \ref{possibileB} in order to ensure that $y_{0, f}^+(u_0) < y_{0,
f}^-(u_0)$ (noticing preliminarily
that \eqref{stimah} is automatically satisfied in view of \eqref{accaC1}).  
We preliminarily observe that $\eqref{accaC1}$ immediately gives
\begin{equation}\label{tesi1}
y_{0, f}^+(u_0) < F^- + \eta.  
\end{equation}
Indeed, if we assume that there exists $u^* \in (0, u_0)$ such that $y_{0, f}^+(u^*)=F^- + \eta$ (we recall that this last quantity is
strictly below $1$ in view of the choice of $\eta$), with $0 < y_{0,
f}^+(u) < F^- + \eta$ for every $u \in (0, u^*)$, by the monotonicity of $R(y)$ we infer that 
\begin{eqnarray*}
y_{0, f}^+(u^*) && = \int_0^{u^*} h'(u) \frac{\sqrt{(y_{0, f}^+(u))(2-y_{0, f}^+(u))}}{1-y_{0, f}^+(u)} \, du - \int_0^{u^* }f(u) \, du  \\
&& < (h(u_0)-h(0)) \frac{\sqrt{(F^- + \eta)(2-F^- - \eta)}}{1-F^- - \eta} + F^-,
\end{eqnarray*}
yielding, in view of $\eqref{accaC1}_i$, the contradiction 
$
y_{0, f}^+(u^*) < F^- + \eta. 
$
\newline
We now focus on $y_{0, f}^-$. In case 1), we have $\min\{F^+, 1\} = F^+$ and 
one can proceed similarly as in Proposition \ref{possibileB}, exploiting the fact that $\bar{y}(u)=\int_u^1 f(s) \,
ds$ is a supersolution for
\eqref{cauchyback}. In view of the monotonicity of $R(y)$, it follows that 
\begin{eqnarray*}
y_{0, f}^-(u_0) && = \int_{u_0}^1 f(u) \, du - \int_{u_0}^1 h'(u) \frac{\sqrt{(y_{0, f}^-(u))(2-y_{0, f}^-(u))}}{1-y_{0, f}^-(u)} \, du \\
&& > F^+ - (h(1)-h(u_0)) \frac{\sqrt{F^+(2-F^+)}}{1-F^+}, 
\end{eqnarray*}
so that $y_{0, f}^-(u_0) > F^- + \eta$ in view of $\eqref{accaC1}_{ii}$ and \eqref{quipos} is fulfilled. 
\newline 
In case 2), we exploit again the same reasoning as in Proposition \ref{possibileB}, inferring first that 
$$
y_{0, f}^-(\xi_0) > 1-\frac{1+\zeta-F^--\eta}{2}. 
$$
Then, the backward solution to 
$$
z'=h'(u) \frac{\sqrt{z(2-z)}}{1-z}, \qquad z(\xi_0)=1-\frac{1+\zeta-F^--\eta}{2}
$$
is again a subsolution for $y_{0, f}^-(u)$ in the interval $[u_0, \xi_0]$. 
Since, for every $w \in (u_0, \xi_0)$, 
$$
z(u)=1-\sqrt{1-(h(u)-h(\xi_0)+\sqrt{1-(1+\zeta-F^--\eta)^2/4})^2},
$$
in view of $\eqref{accaC3}_{i}$ we obtain that $z(u_0) > F^- + \eta$, so that 
$
y_{0, f}^-(u_0) > F^- +\eta
$ 
and the conclusion follows as before.
\end{proof}

\begin{remark}
\textnormal{Similarly as in the case $f \in \mathcal{B}$, if $h'(u) \equiv 0$ then all the assumptions on $h(u)$ are trivially satisfied,
so that Proposition \ref{possibileC}
may be easily adapted to generalize \cite[Proposition 3.9]{GarSan} to the case of density-dependent diffusions. We also observe that
we
may have taken into account more general reaction terms, as in
\cite[Proposition 3.11]{GarSan}, obtaining similar pictures. We omit further details for briefness.}
\end{remark}

In Figures \ref{figC1} and \ref{figC2}, similarly as before, we depict two possible applications of Proposition \ref{possibileC}.

\begin{figure}[!h]
\centering
\includegraphics[width=11cm,height=6cm]{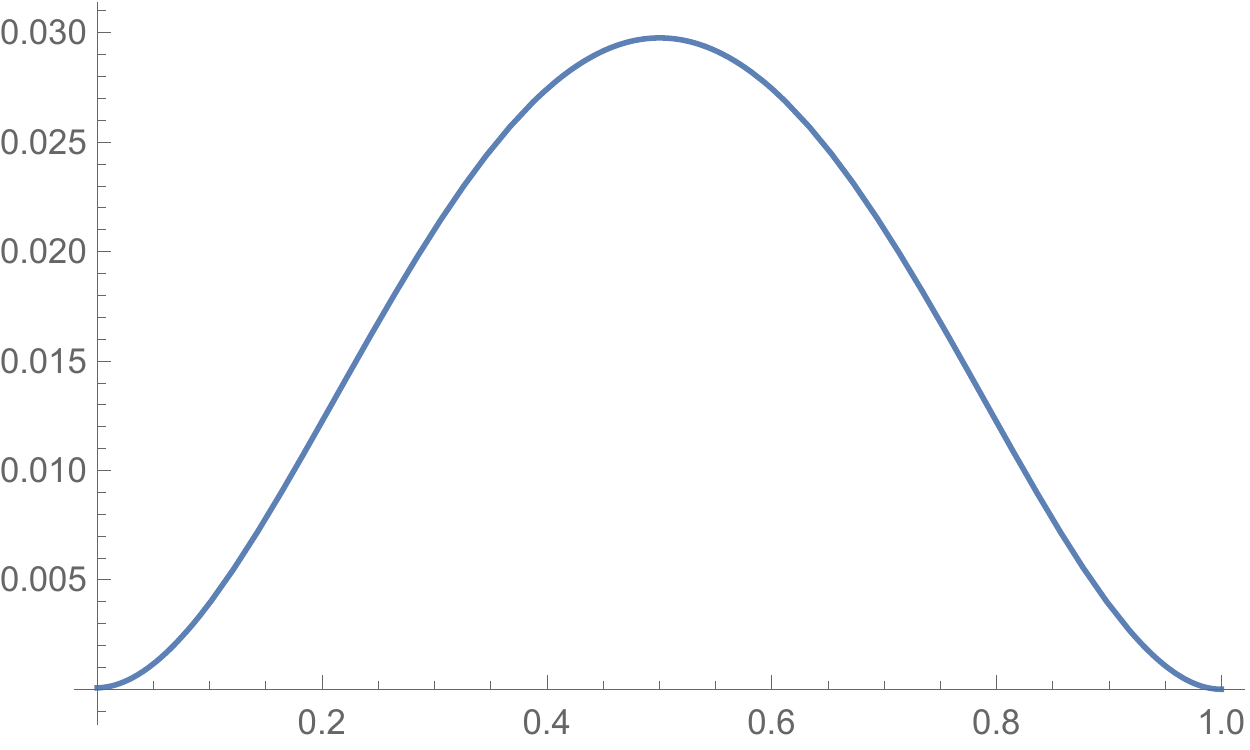}
\caption{\small{As an application of Proposition \ref{possibileC}, we depict the
solution to
$(\star)$ for $h(u)=0.05 u^2$ and $f(u)=2u(1-u)(u-0.4)$: since $F^+ \approx 0.05$ and $F^- \approx 0.017$, we can choose $\eta = 0.015$ to
find the two bounds in $\eqref{accaC1}$ and $\eqref{accaC2}$ approximately equal to $0.05$, so that
$h(0.5)-h(0)=0.0125$ and $h(1)-h(0.5)=0.0375$ enter such conditions. Thus, a positive admissible speed exists; numerically, we find
$c^* \approx 0.151$.}}
\label{figC1}
\end{figure}

\begin{figure}[!h]
\centering
\includegraphics[width=11cm,height=6cm]{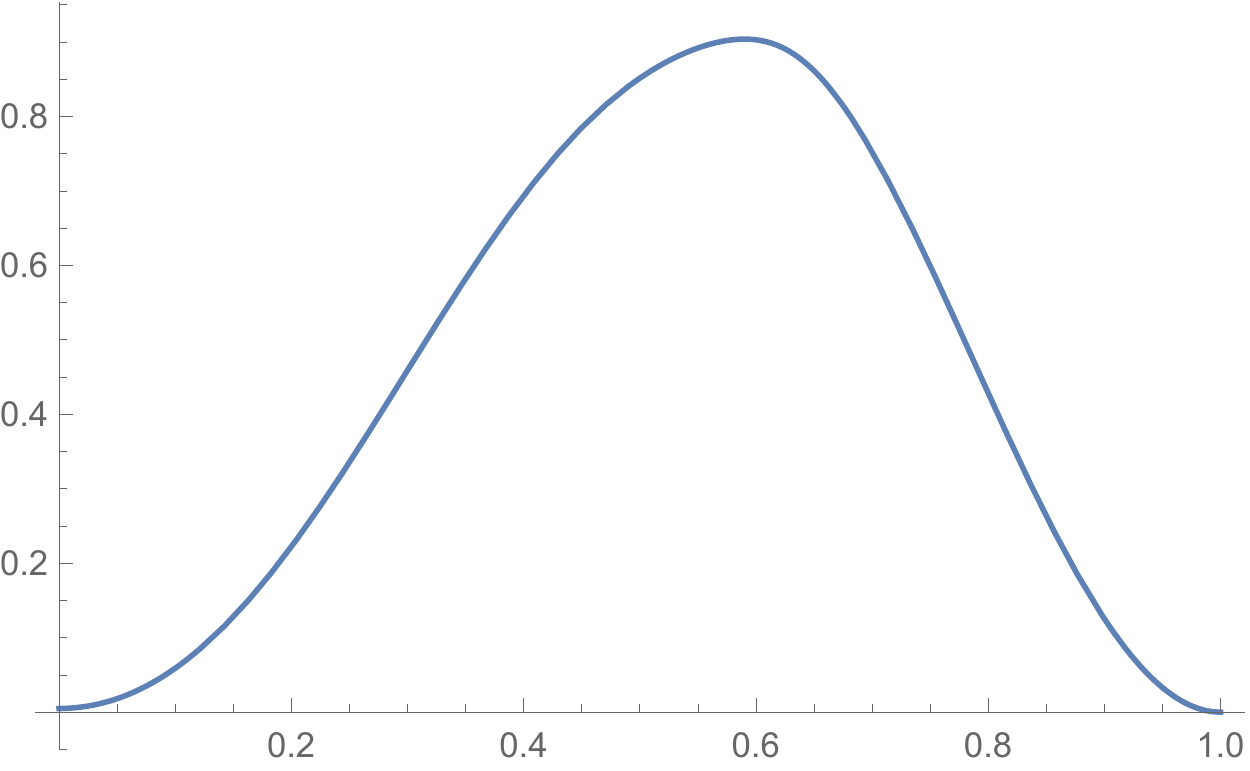}
\caption{\small{Here, maintaining the same convective term, the reaction is instead given by $f(u)=\frac{3}{2} \sin(2\pi u-\pi)(1+2u)$, so that $F^+ > 1$. Similarly as before, we
can check the validity of the assumptions of Proposition
\ref{possibileC}, case 2); numerically, we find $c^* \approx 0.1105$.}}
\label{figC2}
\end{figure}

\section{Negative admissible speeds and nonexistence}\label{sez5}

In this section, we aim at showing some situations when there may be a unique negative admissible speed or even no admissible speeds, when
either $f \in \mathcal{B}$ or $f \in \mathcal{C}$; as before, we will assume that the convection term satisfies assumption $\mathbf{(H')}$. 
We start by observing that, both in the case $f \in \mathcal{B}$ and $f \in \mathcal{C}$, the forward Cauchy problem \eqref{cauchyforward} (i.e., the problem with initial condition $y(0)=0$) has a unique positive solution $y_{c, f}^+$. 
If $f \in \mathcal{B}$, this follows again by integrating the equation, obtaining formula \eqref{inizio}. If $f \in \mathcal{C}$, on the other hand, the solution to \eqref{cauchyforward} itself is unique and positive in a neighborhood of $0$ as a consequence of the monotonicity of $R(y)$, regardless of the sign of $c$.  We have already recalled this fact \cite{BonSan} for $c+h'(0) > 0$ (actually, the same motivation works for $c+h'(0)=0$); if $c+h'(0) < 0$, setting $z(w)=y(D_1-w)$ we obtain the equivalent backward Cauchy problem
\begin{equation}\label{dopocambio}
z'=-\frac{(c+\hat{h'}(D_1-w))}{\hat{d}(D_1-w)}R(z) - \hat{f}(D_1-w), \quad z(D_1)=0,
\end{equation}
which has a unique solution as $R(z)$ is increasing and $-(c+\hat{h'}(D_1-w))/(\hat{d}(D_1-w)) > 0$ (in a left neighborhood of $D_1$). Consequently, $y(w)$ is indeed unique. 
\\
This allows to deduce some monotonicity properties (see again \cite{Har}) of $y_{c, f}^+$: since $z(w)$ decreases on increasing of $c$, $y_{c, f}^+(w)$ increases with $c$ also if $c < 0$. 
In particular, $y_{c, f}^+(D(u_0))$ depends monotonically on $c$ (notice indeed that in this case $y_{c, f}^+(w)$ cannot vanish for $w < D(u_0)$, in view of the sign of its derivative, so that $y_{c, f}^+(D(u_0))$ is well defined). On the other hand, the backward solution $y_{c, f}^-$ is always unique and monotone with respect to $c$ (but may possibly blow-up). 
\\
We can thus make a first observation based on a comparison with
the problem without convection 
\begin{equation}\label{convzero}
\left\{
\begin{array}{l}
\displaystyle x'= \frac{c}{\hat{d}(w)} \frac{\sqrt{x(2-x)}}{1-x}-\hat{f}(w)  \vspace{0.1 cm} \\
x(0)=x(D_1)=0, 
\end{array}
\right.
\end{equation}
assuming the validity of condition \eqref{necessf}. We recall that problem \eqref{convzero} possesses a unique positive admissible
speed $c_f^*$, as it is possible to see proceeding as in the proof of \cite[Proposition 3.9]{GarSan}, with very minor changes due to the
presence of $\hat{d}(w)$. 
\newline 
The following proposition relates $c_f^*$ with the admissible speeds (if any) for problem $(\star)$ through a
first estimate. We denote by $x_{c, f}^+$ (resp. $x_{c, f}^-$) the solution to the forward (resp. backward) Cauchy
problem associated with the differential equation in \eqref{convzero}, with initial condition $x(0)=0$ (resp. $x(D_1)=0$). 
\begin{proposition}
Let $f \in \mathcal{B}$ or $f \in \mathcal{C}$ satisfy \eqref{necessf} and let $h$ fulfill hypothesis $\mathbf{(H')}$. If $c \in \mathbb{R}$
is an admissible speed
for problem $(\star)$,
then it has to hold
\begin{equation}\label{stimarozza}
c_f^* - \max_{u \in [0,1]} h'(u) \leq c \leq c_f^* - \min_{u \in [0,1]} h'(u).
\end{equation}
\end{proposition}
\begin{proof}
In view of the sign of $y'$ and the monotonicity properties of the solutions with respect to $c$, we immediately have that $c$ can be admissible only if $c+\max_{u \in [0, 1]} h'(u) > 0$; moreover, it has to be $c<c_f^*$. Now, if there exists $\sigma > 0$ such that $c < c_f^*-\max_{u \in [0,
1]} h'(u) -\sigma$, then the solution $\bar{y}(w)$ to 
$$
y'=\frac{c+\hat{h}'(w)}{d(w)} \frac{\sqrt{y(2-y)}}{1-y}-\hat{f}(w), \quad y(0)=0
$$
is a subsolution for the forward Cauchy problem 
$$
\left\{
\begin{array}{l}
\displaystyle x'= \frac{c_f^*-\sigma}{\hat{d}(w)} \frac{\sqrt{x(2-x)}}{1-x}-\hat{f}(w)  \vspace{0.1 cm} \\
x(0)=0.
\end{array}
\right.
$$
Since for $c < c_f^*$ we have $x_{c, f}^+(D(u_0)) < x_{c, f}^-(D(u_0))$, by monotonicity the appearance of a $\{0, D_1\}$-connection is then
impossible. \newline 
Analogously, if $c > c_f^*-\min_{u \in [0, 1]} h'(u) + \sigma$, then $y_{c, f}^+(w) > x_{c, f}^+(w)$ and we have $x_{c, f}^+(D(u_0)) >
x_{c,
f}^-(D(u_0))$, so that a $\{0, D_1\}$-connection is again impossible.  
\end{proof}
The control \eqref{stimarozza} gives an effective estimate on the admissible speed for $(\star)$ only once this is already
known to exist and the value of $c_f^*$ is known; however, since it can be easily shown (cf. \cite{HadRot}) that $c_f^*$ is subject to the bound 
$
-2\sqrt{d(u_0) f'(u_0)} < c^*_f <  2\sqrt{d(u_0) f'(u_0)},
$
estimate \eqref{stimarozza} implies, for instance, that 
\begin{equation}\label{stimaconfisher}
-2\sqrt{d_0 f'(u_0)} - \max_{u \in [0,1]} h'(u) \leq c \leq 2\sqrt{d_1 f'(u_0)} - \min_{u \in [0,1]} h'(u).
\end{equation}
Thus, we have a further evidence that there may not be positive admissible speeds if the convection
is too large.  
\smallbreak
In the following, we aim at describing some possible situations we could encounter; since a general statement appears quite delicate to be
given, we limit ourselves to some ``heuristic" descriptions and numerical simulations.

\subsection{Reaction terms of type B} \label{sezneg}

As remarked before, if  $h'(0)=0$
then the admissible speed (if any)
has to be positive, since for $c < 0$ the forward solution $y_{c, f}^+(w)$ is necessarily the zero one 
(for the same reason, any $c \leq -h'(u_0)$ cannot be admissible).
Thus, in this framework nonexistence holds every time that \eqref{quipos} fails;
in Figure \ref{figB0} we illustrate this situation. 
\begin{figure}[!h]
\centering
\includegraphics[width=11cm,height=6cm]{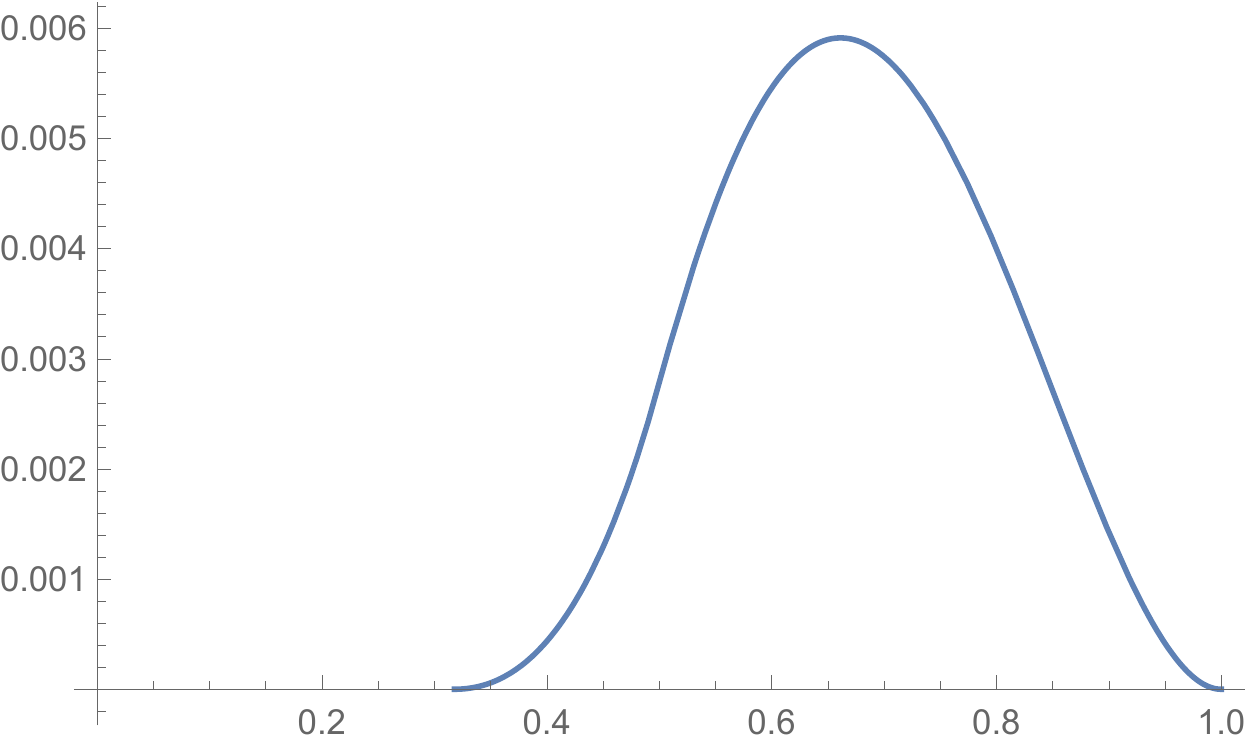}
\caption{\small{We depict the backward solution $y_{0, f}^-$ to the differential equation in $(\star)$, for $c=0$, $h(u)=0.5 u^2$ and
$f(u)=2u(1-u)(\max\{0.5, u\}-0.5)$. Here the convection term is already too strong and makes \eqref{quipos} fail: actually, while the
necessary condition \eqref{stimah} is satisfied, $y_{0, f}^-$ vanishes before reaching
$0$ (indeed, notice that the hypotheses of Proposition \ref{possibileB} cannot be fulfilled for any choice of $\eta$).
Since, on the other hand, negative admissible speeds are prohibited by the fact that $h'(0)=0$, here no admissible speeds exist.}}
\label{figB0}
\end{figure}
\noindent
\newline
Now, assuming that $h'(0) > 0$, there are
essentially two
situations in which the admissible speed may  be
negative:
\begin{itemize} 
 \item[-] \eqref{stimah} holds but \eqref{quipos} fails, so that an
intersection between $y_{c, f}^+$ and $y_{c, f}^-$ may be guaranteed only for a negative value of $c$;
 \item[-]  \eqref{stimah} does not hold (so that the forward solution $y_{0,f}^+(w)$
is not defined in $w=D(u_0)$), but, since
\begin{equation*}
y^+_{c,f}(w):= 1- \sqrt{1-(cD^{-1}(w)+ h(D^{-1}(w))-h(0))^2},
\end{equation*}
there exists a suitable \emph{negative} $\bar c$ such that 
\begin{equation}\label{negativastima}
c u_0 + h(u_0) < 1+h(0), \quad {\rm for \ every \ } c \leq \bar c,
\end{equation}
so that $y_{\bar{c}, f}^+(D(u_0))$ exists and is less than $1$. 
\end{itemize}
The obstacle to the existence of an admissible speed, in both cases, may be the possibility that
$y_{\bar{c}, f}^-(w)$ is not defined for $w \in [D(u_0), 1]$, since $y_{c, f}^-(w)$ has already blown up to $1$
for some $c \in [\bar{c}, 0]$ (since $\bar{c}$ is negative, the blow-up is here possible). 
\newline
In Figures \ref{figB4} and \ref{figB5}, we depict two situations of existence in presence of a convex convective term,  
referring the reader to the considerations in the captions therein. In Figure \ref{figB7}, we present a numerical simulation for a
non-convex convective term, while in Figure \ref{figB8} we show a nonexistence situation.
\begin{figure}[!h]
\centering
\includegraphics[width=11cm,height=6cm]{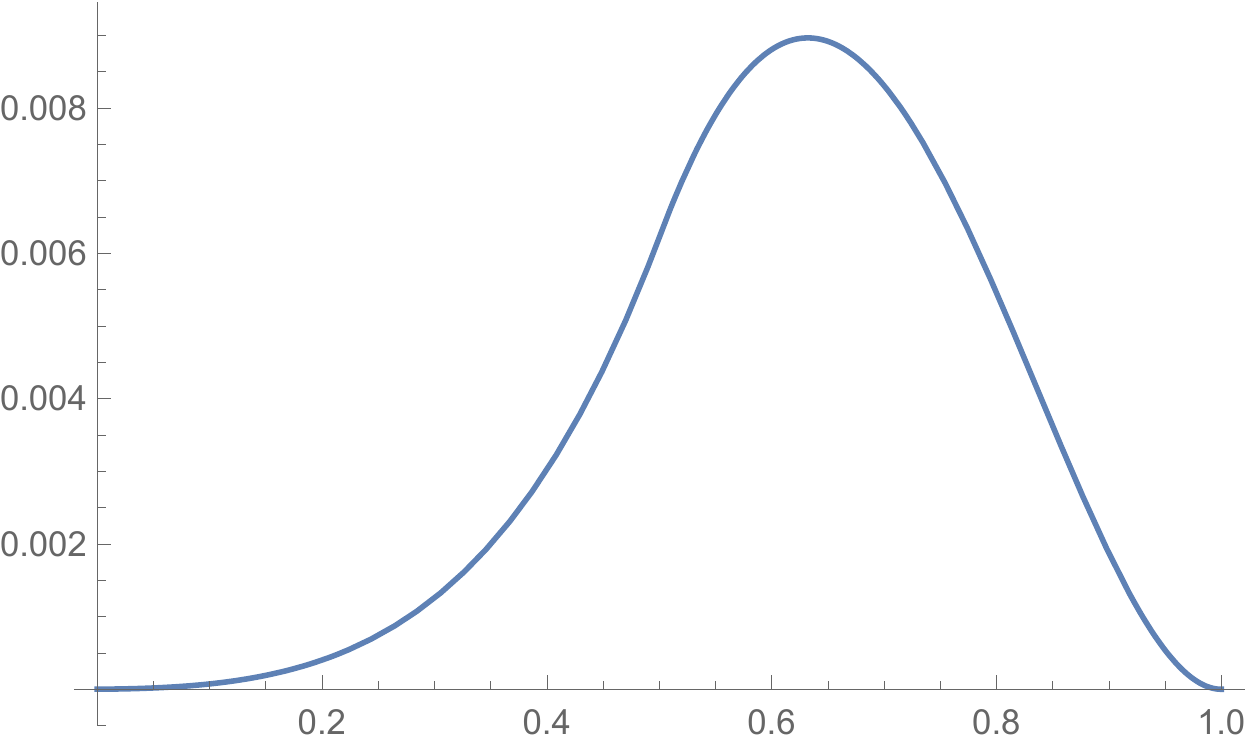}
\caption{\small{In this situation, $h(u)=2\textrm{e}^{t/2}$ and $f(u)=2u(1-u)(\max(t, 0.5)-0.5)$; since $h'(0)> 0$ and
\eqref{stimah} is satisfied, here we expect the existence of a unique negative admissible speed $c^*$ with $c^*+h'(0) >
0$. The numerical simulations show that $c^* \approx -0.9157$.}}
\label{figB4}
\end{figure}

\begin{figure}[!h]
\centering
\includegraphics[width=11cm,height=6cm]{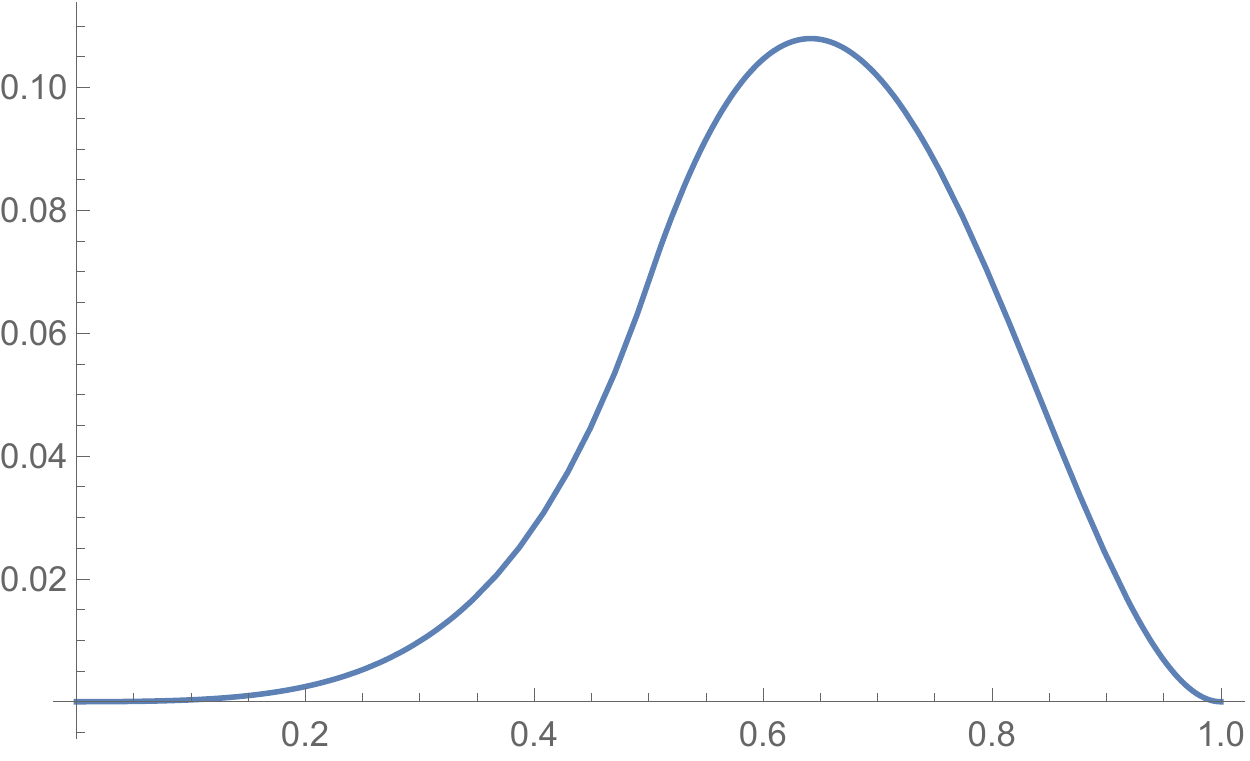}
\caption{\small{Here we represent the second situation described above, namely the convex function $h(u)$ does not
fulfill \eqref{stimah} but a negative admissible speed exists. This is the case for the choices
$f(u)=30u(1-u)(\max\{0.5, u\}-0.5)$ and $h(u)=2\textrm{e}^t$, with corresponding admissible speed $c^* \approx -1.8708$.}}
\label{figB5}
\end{figure}


\begin{figure}[!h]
\centering
\includegraphics[width=11cm,height=6cm]{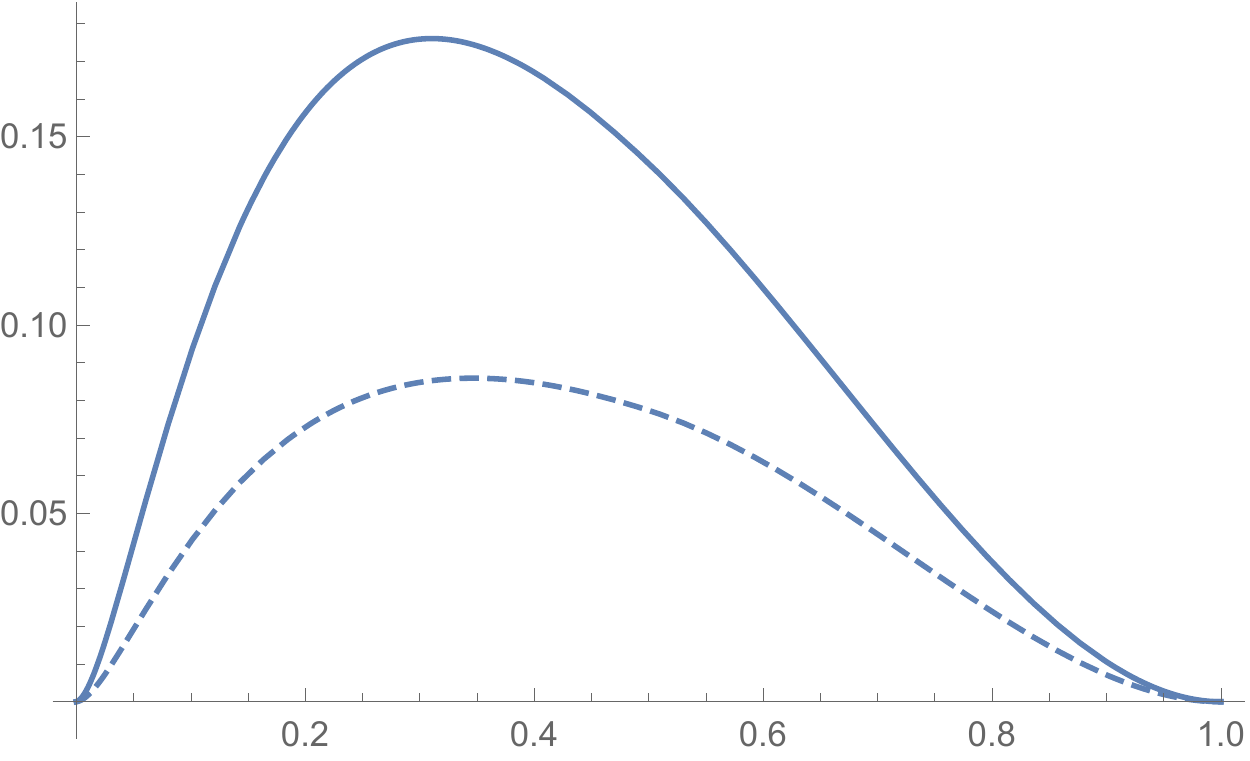}
\caption{\small{In this situation, $f \in \mathcal{B}$ is given by
$f(u)=2u(1-u)(\max\{0.5, u\}-0.5)$ and we have a non-convex convective term given, respectively, by $h(u)=\frac{1}{3}\log(u+0.05)$ (dashed curve) and $h(u)=\frac{1}{2}\log(u+0.05)$
(non-dashed curve). In the first case, condition \eqref{stimah} holds true but \eqref{quipos} fails and we  find $c \approx -0.8399$.
In the second, \eqref{stimah} is not fulfilled, but we are  able to find the desired \{0,1\}-connection for $c^* \approx -1.3866$. The
possibility of having a negative admissible speed strictly relies on the fact that $h'(0) >
0$. }}
\label{figB7}
\end{figure}

\begin{figure}[!h]
\centering
\includegraphics[width=11cm,height=6cm]{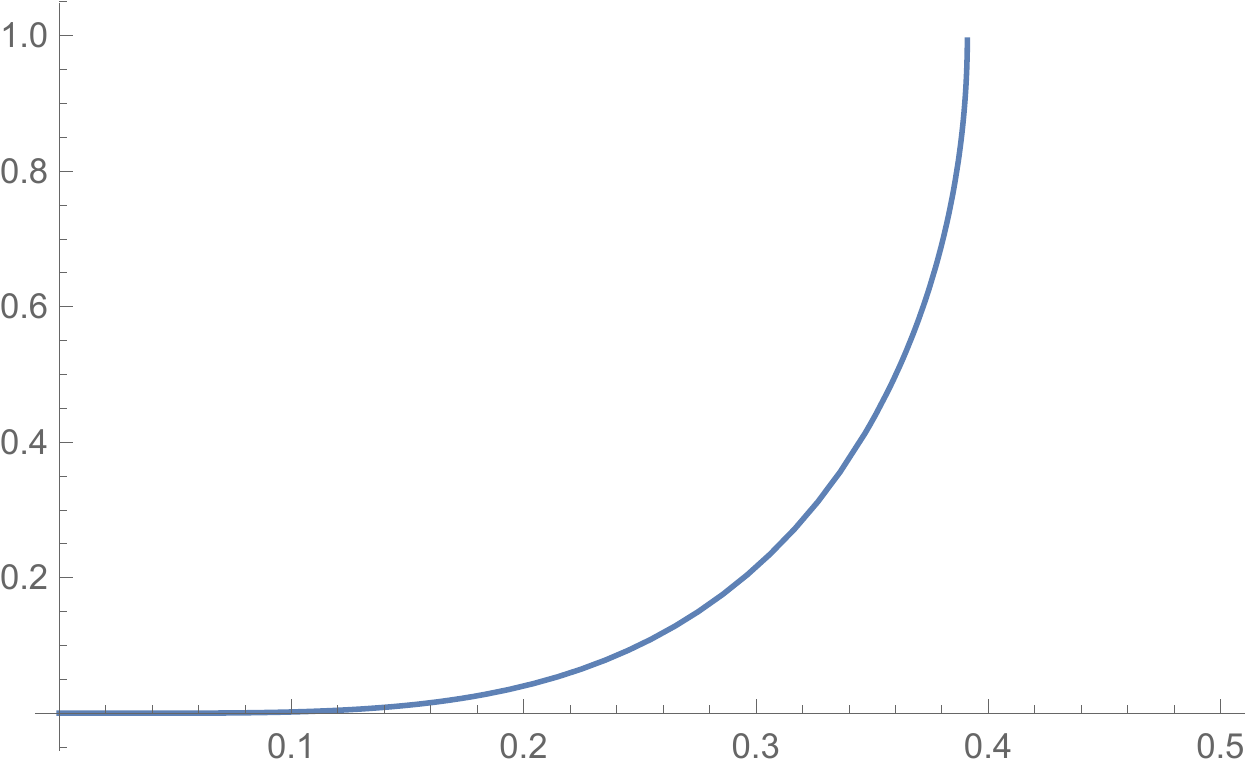}
\vspace{0.1 cm}
\includegraphics[width=11cm,height=6cm]{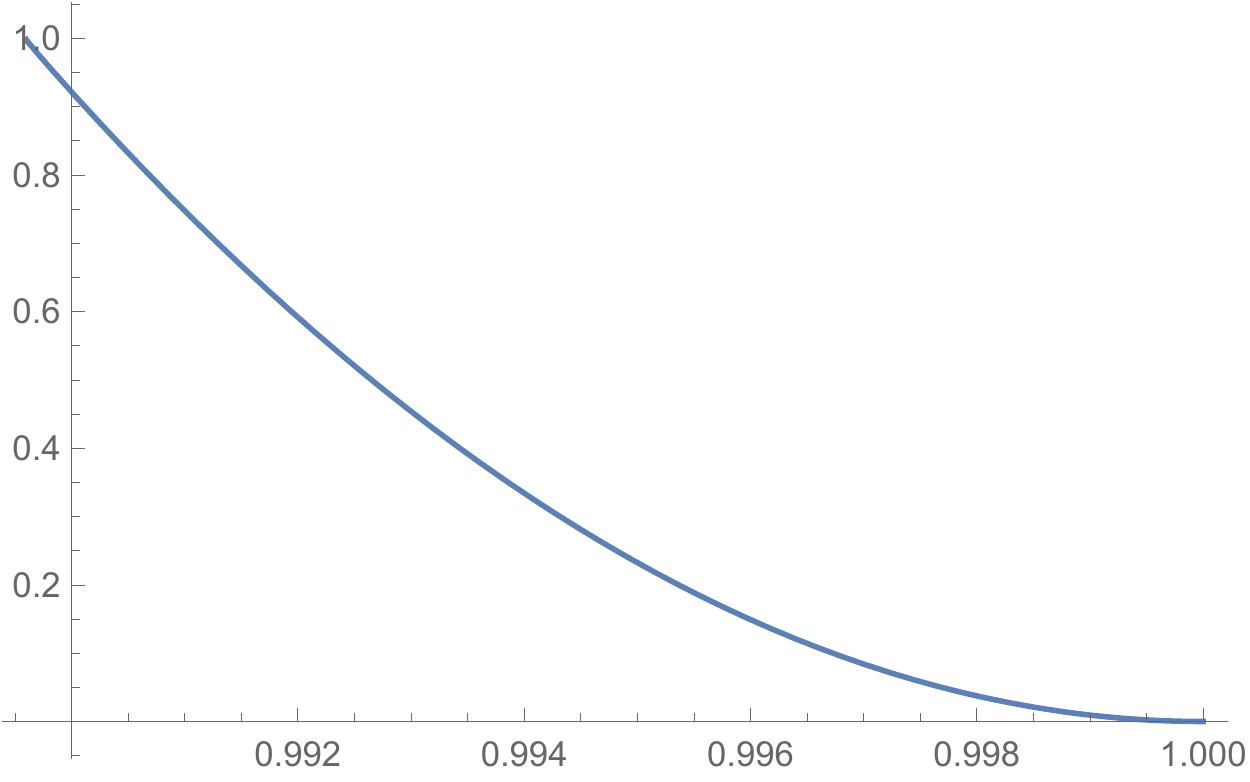}
\caption{\small{Here $c=-1/3$, $h(u)=10 u^2-(20/3)u^3$ and $f(u)$ is equal to zero for $u \in [0, 0.5]$ and equal to $37500 u(1-u)(u-0.5)$ for $u \in [0.5, 1]$. We depict the two solutions $y_{c, f}^+$ (above) and $y_{c, f}^-$ (below), which both blow up to $1$. In view of the monotonicity, here it is not possible to find any admissible speed. Such a portrait persists when $f(u)$ is modified in $[0, 0.5]$ so as to be of type $C$, giving nonexistence also in this case. }}
\label{figB8}
\end{figure}

\subsection{Reaction terms of type C}

Similarly as in the above discussion, also for $f \in \mathcal{C}$ an obstacle to the existence of an admissible speed may be represented by the possibility that $y_{0, f}^+(D(u_0))$ is not
defined (namely, $\max\{\Vert h'\Vert_{L^\infty}, \Vert f^- \Vert_{L^\infty}\}$ is large), so that $c$ has to be chosen sufficiently negative and blow-up from the right is then possible. If this is not the case, the validity of \eqref{quipos}, roughly speaking, will be related with the sign of the admissible speed.  
Intuitively, we may thus expect an ``informal''
picture similar to the following, which takes into account the interplay between the growths of $h'$ 
and $f^-$: 
there exist two constants $M_1, M_2$ with $0 < M_2 < M_1$ such that  
\begin{itemize}
 \item[$\ast$] \,$\max\{\Vert h'\Vert_{L^\infty}, \Vert f^- \Vert_{L^\infty}\} < M_2 \Longrightarrow $ there exists $c^*$ positive;
\vspace{0.1 cm}
 \item[$\ast$] \,$M_2 < \max\{\Vert h'\Vert_{L^\infty}, \Vert f^- \Vert_{L^\infty}\} < M_1 \Longrightarrow $ there exists $c^*$ negative;
\vspace{0.1 cm}
 \item[$\ast$] \,$M_1 \leq \max\{\Vert h'\Vert_{L^\infty}, \Vert f^- \Vert_{L^\infty}\} \Longrightarrow $ there may not exist an admissible
speed $c^*$. 
\end{itemize}
As already mentioned, one should also control the backward solution $y_{c, f}^-$, which may blow up for $w > D(u_0)$.
Since a rigorous proof and an estimate of $M_1$ and $M_2$ in the general case do not appear easy from a theoretical point of view, we just provide some numerical simulations complementing the cases previously taken into account (see Figures \ref{figC3}-\ref{figC4} below). 

\begin{figure}[!h]
\centering
\includegraphics[width=11cm,height=6cm]{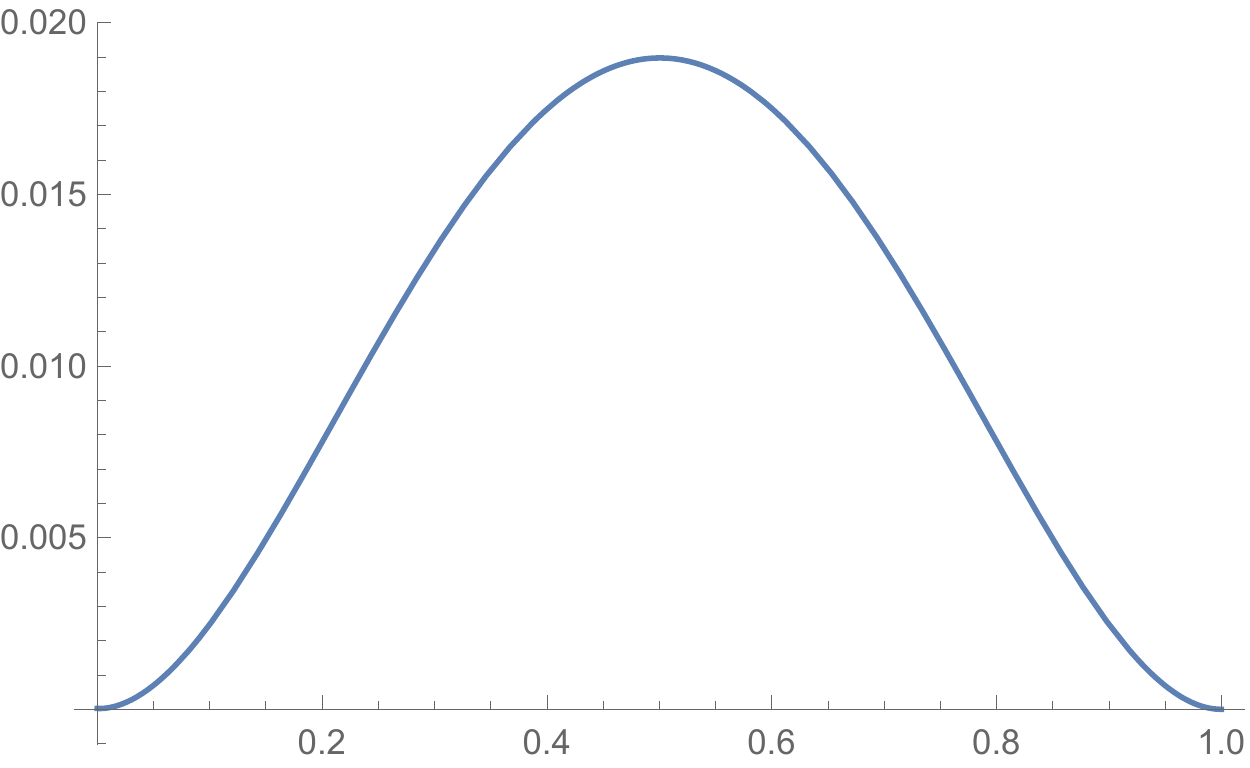}
\caption{\small{We here depict the backward solution $y_{c, f}^-$ to the differential equation in $(\star)$ for $h(u)=0.5 u^2$ and 
$f(u)=2u(1-u)(u-0.4)$, so that conditions \eqref{stimah} and \eqref{necessf} are all satisfied. Nevertheless, the interplay between the
growths of $h'$ and $f$ already gives rise to a negative admissible speed: numerically we find $c^* \approx -0.2462$.}}
\label{figC3}
\end{figure}

\begin{figure}[!h]
\centering
\includegraphics[width=11cm,height=6cm]{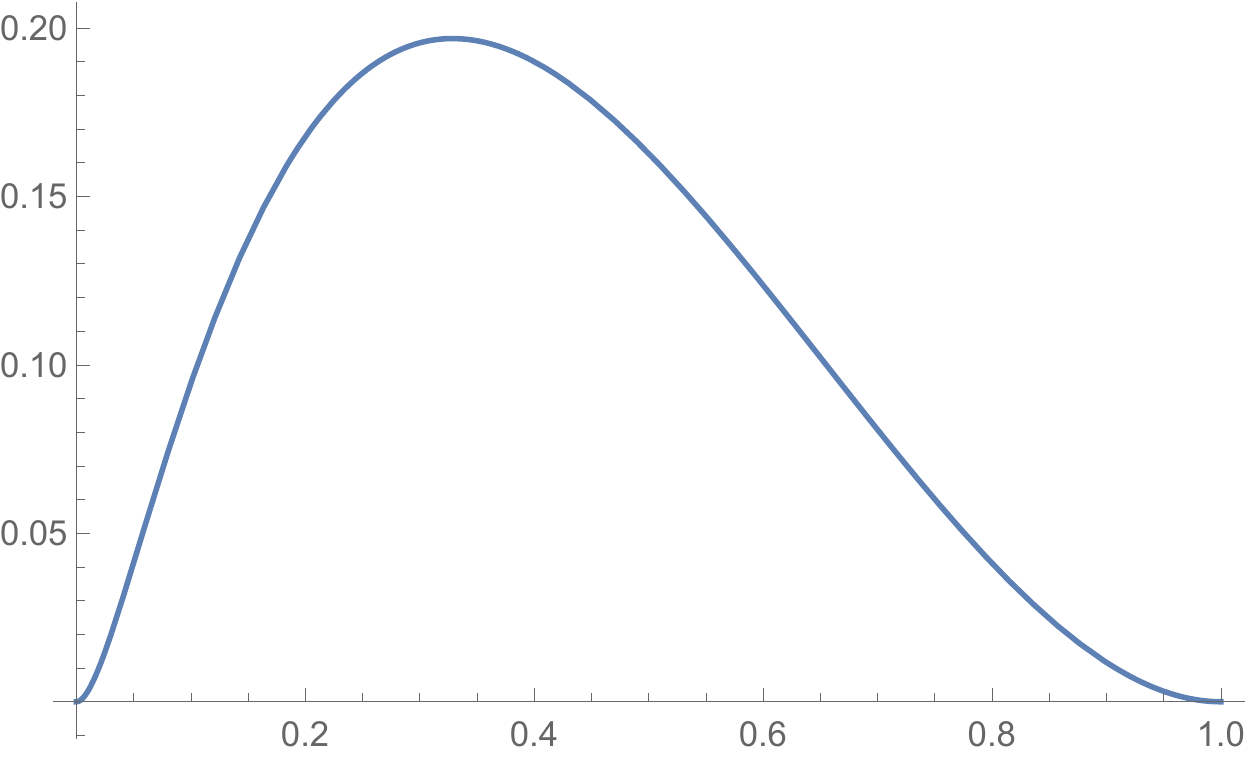}
\caption{\small{In this case, $f(u)=2u(1-u)(u-0.4)$, while $h(u)=\frac{1}{2} \log(t+0.05)$, so that condition
\eqref{stimah} is not satisfied. Nevertheless, the convection does not seem sufficiently strong yet in order for nonexistence of admissible speeds
to hold: the unique admissible speed is indeed found numerically to be negative ($c^* \approx -1.3629$).}}
\label{figC4}
\end{figure}


\section{Admissible speeds in the small viscosity limit -  future perspectives}\label{sez6}

We finally turn our attention to the behavior of the admissible speeds when a small parameter brakes the diffusion, namely
considering the PDE
\begin{equation}\label{PDEepsilon}
\d_t u=\varepsilon^2 \d_x (P[\d_x(D(u))]) -\d_x h(u)+ f(u),
\end{equation}
for some $0<\e \ll 1$. 

First, we deal with the case $D(u) =
u$ and $h(u) \equiv 0$. 
Setting $u(t, x)=v(x+ct)$, we observe that, if
$u$ solves \eqref{PDEepsilon}, then $v$ solves
$$
\varepsilon^2 (P(v'))' - c v' + f(v) = 0,
$$
so that, dividing both sides by $\varepsilon^2$ and setting $b_\varepsilon=c/\varepsilon^2$, we have 
\begin{equation}\label{probv}
(P(v'))' - b_\varepsilon v' + g(v) = 0,
\end{equation}
where 
$$
g(s)=\frac{f(s)}{\varepsilon^2}.
$$
As already remarked, such a problem has a critical speed $b_\varepsilon^*$ such that 
\begin{itemize}
 \item[-] if $f \in \mathcal{A}$, then equation
\eqref{probv} has a
monotone $\{0, 1\}$--connection if and only if
$b \geq b^*_\varepsilon$;
 \item[-] if $f \in \mathcal{B}$, equation \eqref{probv} has a
monotone $\{0, 1\}$--connection if and only if $b = b^*_\varepsilon$.
\end{itemize}
We will not consider explicitly the case $f \in \mathcal{C}$, since it is more delicate in view of assumption \eqref{necessf}. 
If $f \in \mathcal{A}$, we know that
$b^*_\varepsilon=2\sqrt{g'(0)}=2\sqrt{f'(0)}/\varepsilon$; 
consequently, the critical speed for the original problem is given by
$$
c^*_\varepsilon=\varepsilon^2 b^*_\varepsilon=2\varepsilon\sqrt{f'(0)}=\varepsilon c^*, 
$$
where $c^*$ is the critical speed for $\varepsilon=1$. 
\newline
On the other hand, in order to estimate $c^*_\varepsilon$ if $f \in \mathcal{B}$, we
notice that the internal zeros of $f \in \mathcal{B}$ possibly represent equilibria to be asymptotically reached
by a heteroclinic solution. If the uniqueness for the backward Cauchy problems is fulfilled (as is the case here), we thus have
to rule out the speeds yielding other than $\{0, 1\}$--connections. In particular, concerning increasing connections between $u_0$ and $1$, thanks to
Lemma \ref{necessariaBF} it is immediate to deduce that it has necessarily to be $b^*_\varepsilon < 2\sqrt{g'(u_0)}$; indeed,  for values larger than this
threshold, increasing $\{u_0, 1\}$--connections appear. A completely symmetric argument can provide a bound from below, so that 
$$
-2\varepsilon\sqrt{f'(u_0)} < c^*_\varepsilon < 2\varepsilon{\sqrt{f'(u_0)}},
$$
to be compared, in the linear case, with the estimate in \cite{HadRot}.
This ensures that, also in these situations, the admissible speed has order at most equal to $\mathcal{O}(\varepsilon)$. 
\smallbreak
Let us now turn our attention to the general equation
\begin{equation*}
\d_t u=\varepsilon^2\d_x (P[\d_x(D(u))]) - h'(u) \d_x  u + f(u)
\end{equation*}
studied in the present paper, 
with $f$ and $h$ satisfying hypotheses $(\mathbf{F})$ and $(\mathbf{H}')$. 
For the sake of simplicity, we only focus on \emph{positive} admissible speeds; by using the same procedure as before and denoting by $c^*_\varepsilon$ the
critical speed for such equation, on one hand
Theorem
\ref{teorema} ensures that 
\begin{itemize}
 \item[-] if $f \in \mathcal{A}$, 
$$
2\varepsilon\sqrt{d(0) f'(0)} - h'(0) \leq c^*_\varepsilon \leq 2 \varepsilon\sqrt{\sup_{u \in [0, 1]} \frac{d(u)f(u)}{u}} - \min_{u \in [0, 1]} h'(u).
$$
\end{itemize}
In particular, if $f(u) \leq f'(0)u$ for every $u \in [0, 1]$ and $h(u)$ is convex, this means that 
$c^*_\varepsilon=2\e\sqrt{d(0) f'(0)}-h'(0)$.
On the other hand, estimate \eqref{stimaconfisher} states that 
 \begin{itemize}
 \item[-] if $f \in \mathcal{B}$, then $c^*_\varepsilon$ (if it exists) is subject to the following estimate: $$
-2\varepsilon\sqrt{d_0 f'(u_0)}- \max_{u \in [0,1]} h'(u) < c^*_\varepsilon < 2\varepsilon  \sqrt{d_1 f'(u_0)}-\min_{u \in [0,1]} h'(u).
$$
\end{itemize}
Thanks to the previous computations, we observe that if the convective term is convex and satisfies $h'(0)=0$, as for instance in the model case of a
Burgers flux, then the critical speed for reaction terms of type $\mathcal A$ is of order $\e$; similar considerations can be done when $f \in \mathcal B$, when a positive admissible speed exists in dependence on $\varepsilon$ (cf. Propositions \ref{possibileB}
and \ref{possibileC}).

\begin{figure}[!h]
\centering
\includegraphics[width=11cm,height=6cm]{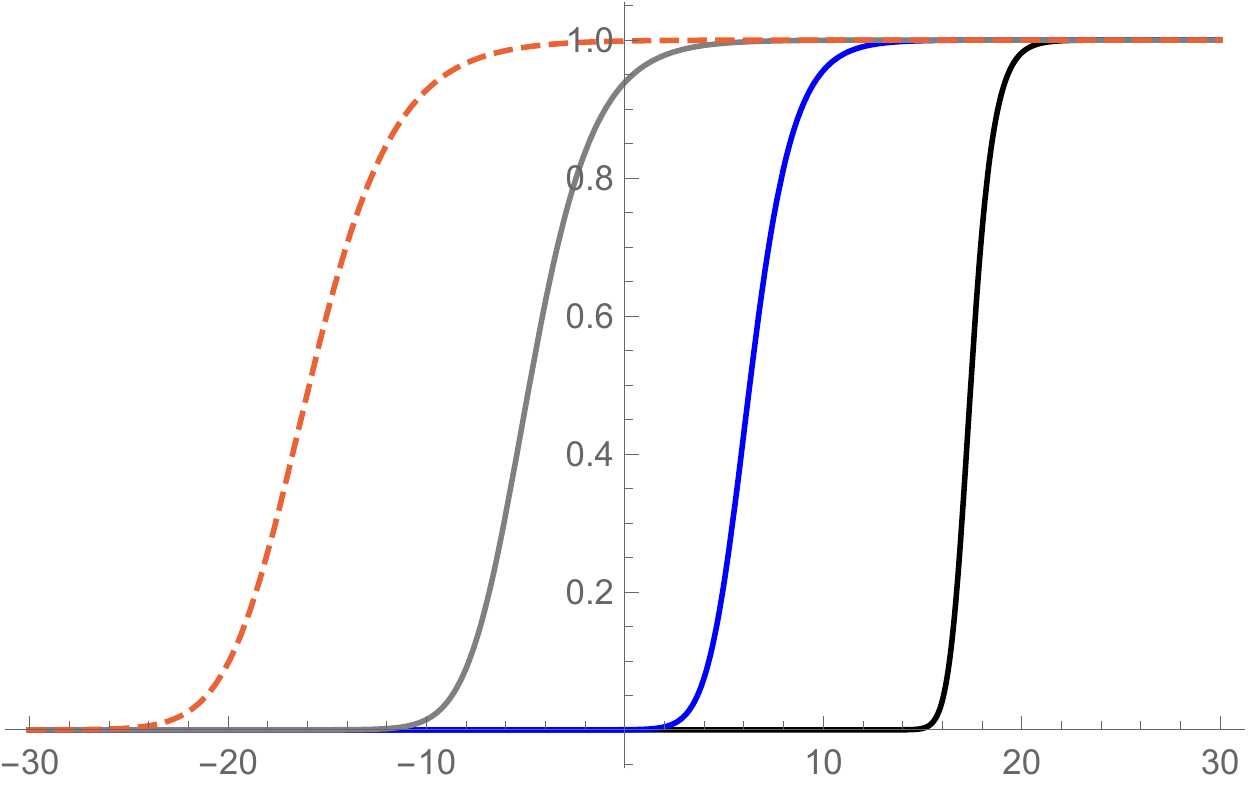}
\caption{\small{Referring to problem \ref{IIord}, we here depict the shape of the traveling waves for $f \in \mathcal{A}$ given by $f(u)=u(1-u)$, $h(u)=0.05u^2$ and $\varepsilon=0.25, 0.5, 0.75, 1$, setting $c=\varepsilon c^*$ (from right to left). Actually, in view of the fact that $\min_{u \in [0, 1]} h'(u)=h'(0)=0$, the observations above ensure that $c^*_\varepsilon=\varepsilon c^*$ and we depict numerically this situation (to this end, we have solved the backward Cauchy problem with initial conditions $y(30)=1-10^{-8}, y'(30)=0)$. We also observe that the shock profile becomes steeper as $\varepsilon \to 0$.}}
\label{fig61}
\end{figure}


\subsection{Issues of stability}

The previous computations suggest that, if we add a small parameter in front of the diffusive term, the speed rate of the heteroclinic
traveling waves 
for \eqref{PDEepsilon} under the additional assumption $h'(0)=0$ is proportional to $\e$, hence it is smaller as $\e$ goes to zero. This is meaningful in light of the
possible appearance of a phenomenon known as \emph{metastability}, whereby the time dependent solutions of an evolutive problem approach
their
stable steady state in an exponentially long time interval as the viscosity coefficient goes to zero.
Such behavior has been extensively studied in the linear diffusion case,  both for reaction-diffusion and advection-diffusion 
type equations, to be considered in a \emph{bounded} interval of the real line. To mention some results concerning this issue, we recall
here
\cite{deGrKara01, KimTzav01, MS} and \cite{ReynWard95} for viscous conservation laws and \cite{BerKamSiv01, Str15, SunWard99}
for Burgers type equations, as well as the fundamental contributions \cite{CarrPego89, FuscHale89} for phase transition problems described
by the Allen-Cahn equation. 
The number of references is very large and underlines how such a problem has arisen a big deal of interest
in the mathematical community over the past years. 
\smallbreak
It may thus be a natural and interesting issue to understand what happens, in this perspective,
when considering a saturating diffusion; one could hopefully expect that, once we reduce our study to a bounded interval $I
\subset \R$ and we complement the equation with appropriate  boundary conditions, a metastable behavior for the time dependent solution
appears.
\smallbreak
Of course, the first issue to be addressed in this direction is  the study of the stability properties of the steady states of the
initial-boundary value problem associated with \eqref{PDEepsilon}. Precisely, given $I=(-\ell,\ell)$, with $\ell > 0$, one has to consider the problem
\begin{equation*}
 \left\{\begin{aligned}
		&\partial_t u	 =\varepsilon^2\d_x (P[\d_x(D(u))]) + \alpha \partial_x h(u)+ \beta f(u),
		&\qquad &x\in I, t > 0,\\
 		&\gamma u(t, \pm \ell) + \delta\d_x u (t, \pm \ell)=0, &\qquad &t >0,\\
		&u(0, x)		 =u_0(x), &\qquad &x\in I,
 	 \end{aligned}\right.
\end{equation*}
where $\alpha, \beta, \gamma$ and $\delta \in \R$, while the functions $P,D,h$ and $f$ satisfy suitable assumptions. 
\smallbreak
Once the existence of a steady state $\bar U(x)$ for such a problem  is proved, there is a broad range of techniques to investigate its
stability/instability.
As usual, one could linearize the original system around the steady state, and subsequently  perform a spectral analysis in order to
study the exact location of the eigenvalues.
Another possible way could be to try to find a Lyapunov functional for the original equation; here the main difficulty is to deal with the sign of the solution at the boundary when computing the time derivative of the Lyapunov functional.
\smallbreak
In possible association with these issues, the stability properties of the traveling waves found in the previous sections may be investigated, as well. 
\smallbreak
To the best of our knowledge, for this problem such directions have not been rigorously studied yet in literature, and this will be
the object
of a future investigation.


\end{document}